\documentclass[reqno,longbibliography]{amsart}

\usepackage{hyperref}
\hypersetup{
	colorlinks   = true, 
	urlcolor     = blue, 
	linkcolor    = purple, 
	citecolor   = blue 
}

\usepackage{amsmath}
\usepackage{amssymb}
\usepackage{amsthm}
\usepackage{amsfonts}
\usepackage{tikz}
\usetikzlibrary{patterns}
\usetikzlibrary{decorations.pathmorphing}

\tikzset{snake it/.style={decorate, decoration=snake}}

\usepackage{enumitem}
\usepackage{xcolor}
\usepackage{mathtools}
\usepackage{stackengine}
\usepackage{comment}
\usetikzlibrary{decorations.pathreplacing,angles,quotes}

\usepackage{float}

\usetikzlibrary{shapes.misc}
\tikzset{cross/.style={cross out, draw, 
         minimum size=2*(#1-\pgflinewidth), 
         inner sep=0pt, outer sep=0pt}}
\usepackage{IEEEtrantools}
\usepackage{url}

\makeatletter
\@namedef{subjclassname@2020}{%
  \textup{2020} MSC}
\makeatother

\mathcode`l="8000
\begingroup
\makeatletter
\lccode`\~=`\l
\DeclareMathSymbol{\lsb@l}{\mathalpha}{letters}{`l}
\lowercase{\gdef~{\ifnum\the\mathgroup=\m@ne \ell \else \lsb@l \fi}}
\endgroup

\newtheorem{theorem}{Theorem}
\newtheorem{proposition}{Proposition}
\newtheorem{lemma}{Lemma}

\theoremstyle{definition}

\newcommand{\details}[1]{}

\newcommand{\Q}{\mathbb{Q}}

\newcommand{\Z}{\mathbb{Z}}
\newcommand{\R}{\mathbb{R}}

\newcommand{\N}{\mathbb{N}}

\newcommand{\mc}{\mathcal}

\newcommand{\dimH}{\dim_{\mathrm H}}

\newcommand{\bs}{\boldsymbol}
\newcommand{\mb}{\mathbb}

\newcommand{\vol}{\textup{vol}}

\DeclarePairedDelimiter{\floor}{\lfloor}{\rfloor}

\title[On the Hausdorff Dim. of Weighted Exactly Approximable Vectors]{ On the Hausdorff Dimension of weighted exactly Approximable Vectors}

\begin{document}

\begin{abstract}
We show that the Hausdorff dimension of \(\bs w\)-weighted \(\tau\)-exactly approximable vectors in $\mb R^d$ coincides with the Hausdorff dimension of \(\bs w\)-weighted \(\tau\)-approximable vectors, generalizing a result of the first named author and De Saxc\'e.
\end{abstract}

\author[P. Bandi]{Prasuna Bandi}
\address{Department of Mathematics, University of Michigan, Ann Arbor}
\email{prasuna@umich.edu}

\author[R. Fregoli]{Reynold Fregoli}
\address{Department of Mathematics, University of Michigan, Ann Arbor}
\email{fregoli@umich.edu}


\thanks{P.B. acknowledges support from AMS-Simons travel grant. }

\maketitle

\section{Introduction}

\noindent
A vector \(\bs w=(w_1,\dots,w_d)\in(0,1)^d\) is called a \emph{weight vector} if
\[
0<w_1\le \dotsb \le w_d \quad\text{and}\quad \sum_{i=1}^d w_i=1.\vspace{-1mm}
\]
For any weight vector \(\bs w\) and $\bs x\in\mb R^d$ set
\[
\|\bs x\|_{\bs w}:=\max\left\{|x_i|^{1/w_i}: i=1,\dotsc,d\right\}.
\]
Let \(\tau>1\) and \(c>0\). The set of \emph{\(\bs w\)-weighted \((c,\tau)\)-approximable} vectors is defined as\vspace{2mm}
\[
W_{\bs w}(c,\tau):=\Big\{\bs x\in[0,1)^d:\|q\bs x-\bs p\|_{\bs w}
< c q^{-\tau}\ \text{for infinitely many }(\bs p,q)\in\mb Z^{d}\times\N\Big\}.\vspace{2mm}
\]
When $c=1$, let \(W_{\bs w}(\tau):=W_{\bs w}(1,\tau)\), and, in the special case \(w_1=\dotsb=w_d=1/d\), let \(W_d(\tau):=W_{\bs w}(\tau)\). \\

Determining the size of the sets \(W_{\bs w}(\tau)\) for \(\tau>1\) is a classical problem in Diophantine approximation.
The Hausdorff dimension of \(W_{\bs w}(\tau)\) was first computed by Rynne in \cite{Ryn98}. We report below \cite[Theorem 1]{Ryn98} in the special case $Q=\mb N$.

\begin{theorem}[Rynne]
\label{thm:Rynne}
For any weight vector $\bs w\in \mb R^d$ and any $\tau>1$
$$\dimH W_{\bs w}(\tau)=\min_{k=1,\dotsc,d}\frac{d+1+\sum_{i=1}^k\tau w_k-\tau w_i}{1+\tau w_k},$$
where $\dimH$ denotes the Hausdorff dimension.
\end{theorem}

A natural extension of this problem regards the dimension of set of \emph{\(\bs w\)-weighted \(\tau\)-exactly-approximable} vectors
\[
E_{\bs w}(\tau):=W_{\bs w}(1,\tau)\setminus\bigcup_{0<c<1} W_{\bs w}(c,\tau).
\]
When the weights $w_i$ are all equal, let us write \(E_d(\tau)\) for \(E_{\bs w}(\tau)\).
The set \(E_d(\tau)\) was first introduced by Jarn\'ik~\cite{Jarnik} in 1931. Jarn\'ik proved that for every $d\geq 1$, the set
\(E_d(\tau)\) is uncountable whenever \(\tau>d\). He further conjectured that the same conclusion should remain true under the weaker condition $\tau>1$.

The precise size of \(E_d(\tau)\) was determined only much later by Bugeaud~\cite{zbMATH02001053}, who proved that \(\dimH E_1(\tau)=\dimH W_1(\tau)\) for all \(\tau>1\).
More refined versions of this result, allowing general approximating functions, were subsequently obtained by Bugeaud~\cite{zbMATH05593660} and by Bugeaud and Moreira~\cite{zbMATH05836403}.



In higher dimension, the problem of establishing the size of the set $E_{d}(\tau)$ remained open for several years. Partial progress was made in \cite{zbMATH07875272} in the case $\tau>d$. The problem was subsequently resolved in full by the first named author and De Saxc\'e \cite{zbMATH08109711}. Their work treats more general approximating functions and, in particular, answers the question of Jarn\'ik by showing $\dimH E_d(\tau)=\dimH W_d(\tau)$ for $\tau>1$.

For further results on the exact approximtion set in other frameworks---such as metric spaces and formal power series---as well as on closely related Diophantine sets, see \cite{zbMATH07656381},\cite{arXiv:2503.06110}, \cite{zbMATH01679547}, \cite{zbMATH07935883}, \cite{zbMATH08012064}\\

In this paper, we are concerned with the classical case of the problem. We extend the result of \cite{zbMATH08109711} to all choices of weights, as follows.

\begin{theorem}
\label{thm:mainthm}
For any weight vector $\bs w\in \mb (0,1)^d$ and any $\tau>1$ it holds that
$$\dimH E_{\bs w}(\tau)=\dimH W_{\bs w}(\tau).$$
\end{theorem}

Since $E_{\bs w}(\tau)\subseteq W_{\bs w}(\tau)$, it is obvious that $\dimH E_{\bs w}(\tau)\leq \dimH W_{\bs w}(\tau)$. Hence, the main difficulty is to prove that $\dimH E_{\bs w}(\tau)\geq \dimH W_{\bs w}(\tau)$. In order to prove this inequality, we shall construct a Cantor set denoted by $E_{\infty}$ contained in $E_{\bs w}(\tau)$ whose Hausdorff dimension is approximately equal to that of $W_{\bs w}(\tau)$.\\

Compared to \cite{zbMATH08109711}, treating arbitrary weights is more demanding and requires two additional ideas. The first one, already appearing in \cite{Ryn98}, is to introduce an auxiliary weight $\tilde{\bs w}$ whose components are approximately equal to $\tau w_i$ and to set up the construction of $E_{\infty}$ by using this new weight vector. 
Second, building on the Cantor-type framework introduced in \cite{zbMATH08109711}, we refine the construction of $E_{\infty}$ by splitting the argument into four regimes. In the equal-weight setting only the first two regimes occur, whereas the genuinely weighted situation forces the appearance of the remaining two (Cases $3$ and $4$ in Section \ref{sec:3.2}). These latter cases are technically more involved, since they require removing an additional family of points to compensate for the extra distortion created by very small weights.


\vspace{2mm}

\textbf{Acknowledgements}:
We are grateful to Ralf Sptazier for his support and encouragement and to Nicolas De Saxc\'e for helpful discussions on the exact approximation problem. 

\vspace{2mm}

\textbf{Notation}: Throughout this article, all boxes in $\R^d$
 are assumed to be axis–parallel. By a box of side length $l_i$, we mean that for $i=1,\dotsc,d$ the side of the box in the $i$-th coordinate direction has length $l_i$.

\section{Auxiliary Results}
\noindent
In this section, we collect several lemmas and theorems required for the proof of Theorem \ref{thm:mainthm}.\\

\noindent
Let $K\subset\mathbb{R}^n$ be a centrally symmetric convex body with nonempty interior, and let $\Lambda\subset\mathbb{R}^n$ be a full-rank lattice. For $1\le i\le n$, the $i$-th successive minimum of $K$ with respect to $\Lambda$ is
\[
\lambda_i(K,\Lambda)
:= \inf\bigl\{ \lambda>0 : \lambda K \cap \Lambda \text{ contains at least $i$ linearly independent vectors}\bigr\}.
\]

\noindent
The following is a classical result (see \cite[Chapter VIII.1, (12) and (13)]{Cas71}).
\begin{theorem}[Minkowski’s Second Theorem on Successive Minima]\label{thm:mink}
For any full-rank lattice $\Lambda\subset \mb R^n$ it holds that
\begin{equation*}
\frac{2^n}{n!}\,\det(\Lambda)
\;\;\le\;\;
\operatorname{vol}(K)\,\prod_{i=1}^n \lambda_i(K,\Lambda)
\;\;\le\;\;
2^n\,\det(\Lambda).
\end{equation*}
In particular, if $\Lambda$ is unimodular (i.e., $\det(\Lambda)=1$), then
\begin{equation}\label{eq:mink}
\frac{2^n}{n!}
\;\;\le\;\;
\operatorname{vol}(K)\,\prod_{i=1}^n \lambda_i(K,\Lambda)
\;\;\le\;\;
2^n .
\end{equation}
\end{theorem}

\noindent
The following lemma is adapted from Lemma~4 of \cite{zbMATH08109711}.
\begin{lemma}\label{lem:pk/qk}
     Fix $\varepsilon>0$, $M>1$ and a weight vector $\bs u\in (0,1)^d$. Let $\bs x\in \R^d$ be such that 
      \begin{equation}\label{eq:bad}
          \|s\bs x-\bs r\|_{\bs u}\geq \frac{\varepsilon}{s} \; \text{ for every } \frac{\bs r}{s}\in \Q^d \text{ with }\; 1\leq s< M
      \end{equation}
     Then for every $\beta>\varepsilon^{-1}$ there exists $\bs p/q\in \Q^d$ such that\vspace{2mm} 
     \begin{enumerate}
         \item $M\leq q\leq M\beta$;\\
         \item $\left|x_i-\dfrac{p_i}{q}\right|\leq \dfrac{1}{M^{1+u_i}\beta^{u_i}} \quad \mbox{for }i=1,\dotsc,d$. 
     \end{enumerate}
\end{lemma}
 \begin{proof}
     Let $Q=M\beta$. By Minkowski's first theorem, there exists $\frac{\bs p}{q}\in \Q^d$ with $1\leq q\leq Q=M\beta$ such that
     \begin{equation}\label{w-dir}
       \left|x_i-\frac{p_i}{q}\right|\leq \frac{1}{q Q^{u_i}}=\frac{1}{qM^{u_i}\beta^{u_i}} \quad \mbox{for }i=1,\dotsc,d.  
     \end{equation}
     If $q<M$, then by \eqref{eq:bad} we have 
     $$\|q\bs x-\bs p\|_{\bs u} \geq \frac{\varepsilon}{q}>\frac{\varepsilon}{M}$$
     which implies that for some $i=1,\dotsc,d$
     $$\left|x_i-\frac{p_i}{q}\right|> \dfrac{\varepsilon^{u_i}}{qM^{u_i}}>\dfrac{1}{qM^{u_i}\beta^{u_i}},$$
      in contradiction to \eqref{w-dir}. Hence, we conclude that $q\geq M$, thus proving the first part. Using $q\geq M$ in \eqref{w-dir}, we obtain the second part of the proposition. 
     \end{proof}

The following lemma is a variant of Simplex Lemma originating in the works of Davenport and Schmidt \cite[page~57]{Schmidt}.
\begin{lemma}\label{lem:hyp}
   Fix parameters $R>1$, $n\in \N$ and a weight vector $\bs u$. Let $E$ be a box whose side length is 
   $\leq \rho_0^{(i)}R^{-(1+u_i)n} \text{ where } \rho_0^{(i)}=\frac{R^{-(1+u_i)}}{(d+1)!}$. Then for every $0<\varepsilon < R^{-1}((d+1)!)^{-1/u_1}$, the set
    $$S(E,\varepsilon):=\left\{\frac{\bs r}{s}\in \Q^{d}: R^n\leq s<R^{n+1} \text{ and } \|s\bs x-\bs r\|_{\bs u}\leq \frac{\varepsilon}{s} \text{ for some } \bs x\in E\right\}$$ lies on an affine plane in $\R^{d}$.
\end{lemma}
\begin{proof}
    Let $\bs y\in \R^d$ be the center of the box $E$. 
    Define $$u_{\bs y}:=\begin{pmatrix}
    I_{d \times d} & \bs y_{d \times 1} \\[6pt]
    1              & 0
\end{pmatrix}
$$
    Denote by $\Lambda_{\bs y}$, the unimodular lattice $u_{\bs y}\Z^{d+1}$ in $\R^{d+1}$.
Let
$$S'(E,\varepsilon):=\left\{(\bs r,s)\in \Z^{d+1}: R^n\leq s<R^{n+1} \text{ and } \|s\bs x-\bs r\|_{\bs u}\leq \frac{\varepsilon}{s} \text{ for some } \bs x\in E\right\}$$
Suppose $(\bs r,s)\in S'(E,\varepsilon) $. Then
\begin{align*}
    |sy_i-r_i|& \leq  s|y_i-x_i|+|sx_i-r_i|\\
    &\leq R^{n+1} \rho_0^{(i)}R^{-(1+u_i)n}+ \frac{\varepsilon^{u_i}}{R^{nu_i}}=R^{-nu_i}(R\rho_0^{(i)}+\varepsilon^{u_i})\leq \frac{2\cdot R^{-(n+1)u_i}}{(d+1)!}
\end{align*}
The last inequality above follows since $\varepsilon < \frac{R^{-1}}{((d+1)!)^{1/u_1}}$ and $\rho_0^{(i)}=\frac{R^{-(1+u_i)}}{(d+1)!}$.\\
Hence, $u_{\bs y} \cdot S'(E,\varepsilon)\subseteq K \cap \Lambda_{\bs y}$ where
$$K:= \prod_{i=1}^d \left[-\frac{2 \cdot R^{-(n+1)u_i}}{(d+1)!}, \frac{2\cdot R^{-(n+1)u_i}}{(d+1)!}\right] \times [-R^{(n+1)},R^{(n+1)}].$$
Now, note that $$\vol K < \frac{2^{d+1}}{(d+1)!}.$$
Hence, by the lower bound in  Theorem \ref{thm:mink} $u_{\bs y} \cdot S'(E,\varepsilon)$ cannot contain $d+1$ linearly independent vectors. Therefore, $S'(E,\varepsilon)$ is contained in a hyperplane in $\R^{d+1}$ which implies that $S(E,\varepsilon)$ is contained in an affine plane in $\R^d$. 
\end{proof}
\noindent
The following lemma is inspired by \cite[Lemma 1]{Ryn98}.

\begin{lemma}
\label{lem:Rynne}
For every weight vector $\bs w$ and every $\tau>1$ there exist an integer $K=K(\bs w,\tau )$, $0\leq K<d$, and a real number $\delta_0(\bs w, \tau)>0$ such that for any $0<\delta\leq\delta_0$ there exist vector $\tilde{\bs w}\in\mb (0,+\infty)^d$ with the following properties:\vspace{2mm}
\begin{enumerate}[label={($\tau$\arabic*)}, ref={$\tau$\arabic*}]
  \item \label{tau1}
    $\tilde{w}_i=\tau w_i-\delta\,(1+\tau w_i)$ for $i=1,\dotsc,K$;\vspace{2mm}
  \item \label{tau2}
    $\tilde w_{K}\le \tilde{w}_{K+1}=\dotsb=\tilde w_{d}
    < \tau w_{K+1}-\delta\,(1+\tau w_{K+1})$;\vspace{2mm}
    \item \label{tau3} $\sum_{i=1}^d\tilde w_i=1$.\vspace{2mm}
\end{enumerate}
In particular, $\tilde w_1\leq\dotsb \leq \tilde w_d$.
\end{lemma}

\begin{proof}
If $\tau w_1> 1/d$, we put $K=0$ and we choose $\delta_0$ so small that $\tau w_1-\delta_0\cdot (1+\tau w_1)\geq 1/d$. Then it is enough to set $\tilde w_i=1/d$ for $i=1,\dotsc,d$. Now, let us assume that $\tau w_1\leq 1/d$. Fix $\delta_0<(\tau-1)/(\tau +d)$ so small that
$$0<\tau w_i-\delta_0\cdot (1+\tau w_i)\leq \tau w_{i+1}-\delta_0\cdot (1+\tau w_{i+1})\quad\mbox{for all }i=1,\dotsc,d-1.$$
Then for fixed $0<\delta<\delta_0$ consider the expression
\begin{equation}
\label{eq:wtildedef}
\sum_{i=1}^h \tau w_i-\delta\cdot (1+\tau w_i) +(d-h)(\tau w_{h+1}-\delta\cdot (1+\tau w_{h+1})).    
\end{equation}
By our choice of $\delta_0$, this expression is non-decreasing in $h$, strictly less than $1$ for $h=0$, and strictly greater than $1$ for $h=d-1$. Choose $K$ to be the minimum integer $h$ between $1$ and $d-1$ such that \eqref{eq:wtildedef} is strictly greater than $1$. Then we have that
\begin{multline*}
\sum_{i=1}^{K}\tau w_i-\delta\cdot(1+\tau w_i)+(d-K)(\tau w_{K}-\delta\cdot (1+\tau w_K))\\
=\sum_{i=1}^{K-1}\tau w_i-\delta\cdot(1+\tau w_i)+(d-(K-1))(\tau w_{K}-\delta\cdot (1+\tau w_{K}))\leq 1\\
<\sum_{i=1}^{K}\tau w_i-\delta\cdot(1+\tau w_i)+(d-K)(\tau w_{K+1}-\delta\cdot(1+\tau w_{K+1})).    
\end{multline*}
Thus, there exists $\tau w_K-\delta\cdot(1+\tau w_K)=\tilde w_K\leq x <\tau w_{K+1}-\delta\cdot(1+\tau w_{K+1})$ such that
$$\sum_{i=1}^{K}\tau w_i-\delta\cdot(1+\tau w_i)+(d-K)\cdot x=1.$$
On choosing $\tilde w_{K+1}=\dotsb =\tilde w_d=x$ the proof is concluded. Note that, if $\delta_0$ is sufficiently small in terms of $\bs w$ and $\tau$, then $K$ is independent of $\delta>0$.
\end{proof}

\section{Construction of the Cantor set}

Our goal is to construct a Cantor set $E_{\infty}$ such that $E_{\infty}\subseteq E_w(\tau)$. This will require us to introduce several auxiliary constants, which we do in this section.

Let $\tau>1$, $d\geq 2$ and $\bs w\in (0,1)^d$ be a weight vector. Let $0\leq K<d$ and $\delta_0>0$ be as in Lemma \ref{lem:Rynne}. Fix $0<\delta< \min\{\delta_0,\tau w_1/(1+\tau w_1)^{-1}\}$ and let $\tilde{\bs w}\in (0,1)^d$ be the weight vector associated with $\delta$, $\tau$, and $\bs w$ coming from Lemma \ref{lem:Rynne}. 
Define 
$$\alpha:=\max\left\{\frac{\tau w_i}{\tilde{w}_i}: 1\leq i\leq d \right\}\geq 1,$$
and 
\[
\alpha' \coloneqq
\begin{cases}
\displaystyle \min\{\, w_k - w_1 : 1 \le k \le d,\; w_k \neq w_1 \,\}, 
& \text{if there exists } 1\leq k \leq d\text{ with } w_k \neq w_1, \\[6pt]
1, & \text{otherwise.}
\end{cases}
\]

\noindent
Let $R>(8(d+1)!)^{\frac{1}{1+\tilde{w}_1}}$. We define sequences $n_k \rightarrow \infty, \varepsilon_k \rightarrow 0$ and  $c_k\rightarrow 1$ inductively as follows. Let $\varepsilon_0:=R^{-2(1+\frac{1}{\tilde{w}_1})}(d+1)!^{-1/\tilde{w}_1}$. Choose $\frac{3}{4}<c_k<1$ such that $1-c_k>\varepsilon_{k-1}$ and choose $n_k\in \N$ such that 
\begin{equation}\label{eq:nk}
 n_k\geq -\max\left\{\frac{2\frac{1+\tau w_1}{1+\tilde{w}_1}+d+1}{\alpha'(\tau w_1-\tilde{w}_1)(1-\frac{1}{1+\tilde{w}_1})},4,\frac{2}{\tau-1} \right\}\log_R \varepsilon_{k-1}   
\end{equation}
for all $k\in\N$. Set $n_0^{(d)}=0$ and for $k\in \N$ and $i=1,\dotsc,d$ define
\begin{equation}\label{eq:nki}
    n_k^{(i)}:=\floor*{\frac{1+\tau w_i}{1+\tilde{w}_i}(n_k-2\log_{R}\varepsilon_{k-1})}+1. 
\end{equation}
and
\begin{equation}
\label{eq:epsk}
\varepsilon_k:=2^{-1/(w_1\tilde w_1)}R^{-(\alpha (n_k^{(d)}+1)-n_k)} 
\end{equation}
Set also for $1\leq i\leq d$
$$\rho_0^{(i)}:=\frac{R^{-(1+\tilde{w}_i)}}{(d+1)!}\quad \text{and}\quad \rho_i:=R^{1+\tilde{w}_i}$$
Define
$$\rho:=\prod_{i=1}^d \floor*{R^{1+\tilde{w}_i}}$$

\subsection{Some technical lemmas}
In this subsection we prove several lemmas which will be useful in the sequel. These are largely technical and may be skipped on a first read.

\begin{lemma}
\label{lem:newlem1}
There exists a constant $\xi_0$ depending only on $\delta, \tau,\bs w,$ and $\tilde{\bs w}$, such that for every $\xi\geq \xi_0$ it holds that
$$\rho_0^{(1)}R^{-(1+\tilde w_1)\xi n_k^{(d)}}<R^{-(1+\tau w_d)(n_k^{(d)}+1)}$$
for every $k\in\mb N$.
\end{lemma}

\begin{proof}
Take $\xi_0=(1+\tau w_d)/(1+\tilde w_1)+1$. 
\end{proof}

From now on, we will assume additionally that $\xi\geq \xi_0$ is fixed and that $n_{k+1}>\xi n_{k}^{(d)}$ for all $k\in\mb N$.

\begin{lemma}
\label{lem:lemma5}
Equations \eqref{eq:nk}, \eqref{eq:nki}, and \eqref{eq:epsk} imply that
$$n_k\leq n_k^{(1)}\leq n_k^{(2)}\leq n_k^{(3)}\leq \dots \leq n_k^{(d)} <{\xi} n_k^{(d)}<n_{k+1}$$
for every $k\in\mb N$. Moreover, the following conditions hold
\begin{equation}
\label{eq:wave}
R^{n_k^{(1)}}
\leq R^{(1+\tau w_1-\tilde{w}_1)n_k}\varepsilon_{k-1}^d,    
\end{equation}
\begin{equation}
\label{eq:ckcond}
R^{-(1+\tilde{w}_i)n_{k}^{(i)}}\leq (1-c_k)\frac{\varepsilon_{k-1}^{1+\tau w_i}}{R^{n_k(1+\tau w_i)}}\quad\mbox{for }i=1,\dotsc,d,    
\end{equation}
\begin{equation}
\label{eq:nkdup}
n_k^{(d)}\leq 2\cdot \frac{1+\tau w_d}{1+\tilde w_d}n_k,    
\end{equation}
\begin{equation}\label{eq:volnpi}
  \max_i\left\{\frac{R^{-(1+\tau w_i)n}}{\rho_0^{(i)}R^{-(1+\tilde{w}_i)\max\{n,n_k^{(i)}\}}} \right\}=\frac{R^{-(1+\tau w_1)n}}{\rho_0^{(1)}R^{-(1+\tilde{w}_1)\max\{n,n_k^{(1)}\}}}\quad\mbox{for }n\geq n_{k}^{(1)}.
\end{equation}
\end{lemma}
\begin{proof}
By \eqref{tau1}, we have that for $1\leq i\leq K$,
$$\tilde{w}_i=\tau w_i-\delta(1+\tau w_i)=\tau w_i(1-\delta)-\delta.$$
Since $w_i\leq w_{i+1}$, we deduce that $\tilde{w}_i\leq \tilde{w}_{i+1}$ for $1\leq i\leq K-1$.
Hence, from \eqref{tau2} it follows that
$$0<\tilde{w}_1\leq \dots \leq \tilde{w}_K\leq  \tilde{w}_{K+1}=\tilde{w}_{K+2}= \dots =\tilde{w}_d$$
Now, by $\eqref{tau1}$, we also have that 
$$\frac{1+\tau w_i}{1+\tilde{w}_i}=\frac{1}{1-\delta} \text{ for } 1\leq i\leq K,$$
and, by \eqref{tau2}, we have that
$$\frac{1+\tau {w}_{K+1}}{1+\tilde{w}_{K+1}}\leq \frac{1+\tau {w}_{K+2}}{1+\tilde{w}_{K+2}}\leq \dots \leq \frac{1+\tau {w}_{d}}{1+\tilde{w}_d}.$$
Since $\tilde{w}_{K+1}<\tau w_{K+1}-\delta(1+\tau w_{K+1})$, we also have that
$$\frac{1+\tau w_{K+1}}{1+\tilde{w}_{K+1}}>\frac{1}{1-\delta}=\frac{1+\tau w_{K}}{1+\tilde{w}_K}.$$
Combining the above inequalities, we conclude that 
$$1<\frac{1}{1-\delta}=\frac{1+\tau {w}_1}{1+\tilde{w}_1}=\dots=\frac{1+\tau {w}_{K}}{1+\tilde{w}_K}<\frac{1+\tau {w}_{K+1}}{1+\tilde{w}_{K+1}}\leq \dots \leq \frac{1+\tau {w}_{d}}{1+\tilde{w}_d}.$$
Then \eqref{eq:nki} implies that
$$n_k<n_k^{(1)}=\dots= n_k^{(K)}<n_k^{(K+1)}\leq \dots\leq  n_k^{(d)}.$$
To prove \eqref{eq:wave}, just observe that
$$R^{n_k^{(1)}}\overset{\eqref{eq:nki}}{\leq}R^{\frac{1+\tau w_1}{1+\tilde{w}_1}n_k}\varepsilon_{k-1}^{-2\frac{1+\tau w_1}{1+\tilde{w}_1}-1}\overset{\eqref{eq:nk}}{\leq}R^{(1+\tau w_1-\tilde{w_1})n_k}\varepsilon_{k-1}^d.$$
To prove \eqref{eq:ckcond}, note that 
by \eqref{eq:nki}, it holds that
$$n_k^{(i)}\geq \frac{1+\tau w_i}{1+\tilde{w}_i}(n_k-2\log_{R}\varepsilon_{k-1}),$$
which gives
$$R^{-(1+\tilde{w}_i)n_k^{(i)}}\leq R^{-n_k(1+\tau w_i)}\varepsilon_{k-1}^{2(1+\tau w_i)}\leq (1-c_k)\frac{\varepsilon_{k-1}^{1+\tau w_i}}{R^{n_k(1+\tau w_i)}},$$
where in the last inequality we used the fact that $\varepsilon_{k-1}^{1+\tau w_i}\leq \varepsilon_{k-1}\leq  1-c_k$.\\
To prove \eqref{eq:nkdup}, observe that, by \eqref{eq:nk},
$n_k\geq -4\log_R\varepsilon_{k-1}$. Hence, 
\begin{equation*}
    n_k^{(d)}\overset{\eqref{eq:nki}}{=}\floor*{\frac{1+\tau w_d}{1+\tilde{w}_d}(n_k-2\log_{R}\varepsilon_{k-1})}+1 \leq 2\frac{1+\tau w_d}{1+\tilde{w}_d}n_k
\end{equation*}
Let us now prove \eqref{eq:volnpi}. For fixed $1\leq j\leq d$ we aim to show that
$$\frac{R^{-(1+\tau w_j)n}}{R^{-(1+\tilde{w}_j)\max\{n,n_k^{(j)}\}}}\leq \frac{\rho_0^{(j)}}{\rho_0^{(1)}}R^{-(\tau w_1-\tilde{w}_1)n} \quad\text{ for }n\geq n_k^{(1)}.$$
First, consider the case $n\geq n_k^{(j)}$. Then we need to prove that
$$R^{-(\tau w_j-\tilde{w}_j)n}\leq \frac{\rho_0^{(j)}}{\rho_0^{(1)}}R^{-(\tau w_1-\tilde{w}_1)n}$$
When $\tau w_j-\tilde w_j\neq \tau w_1-\tilde w_1$, this follows easily by \eqref{eq:nk}. If $\tau w_j-\tilde{w}_j=\tau w_1-\tilde{w}_1$ then it must be $\tilde{w}_j=\tilde{w}_1$ (an easy consequence of \eqref{tau1} and \eqref{tau2}) and the inequality reduces to $1\leq 1$. Now, suppose that $n_k^{(1)}\leq n<n_k^{(j)}$. Then we need to show that
$$R^{(1+\tilde{w}_j)n_k^{(j)}}\leq \frac{\rho_0^{(j)}}{\rho_0^{(1)}}R^{(1+\tilde{w}_1+\tau(w_j-w_1))n}.$$
Since $n_k^{(j)}\overset{\eqref{eq:nki}}{\leq} (1+\tau w_j)(1+\tilde{w}_j)^{-1}(n_k-2\log_{R}\varepsilon_{k-1})+1 $ and $n\geq n_k^{(1)}\overset{\eqref{eq:nki}}{\geq} (1+\tau w_1)(1+\tilde{w}_1)^{-1}(n_k-2\log_{R}\varepsilon_{k-1})$, it is enough to pove that
$$R^{\tau(w_j-w_1)(n_k-2\log_{R}\varepsilon_{k-1})+(1+\tilde{w}_j)}\leq \frac{\rho_0^{(j)}}{\rho_0^{(1)}} R^{\tau(w_j-w_1)\frac{1+\tau w_1}{1+\tilde{w}_1}(n_k-2\log_{R}\varepsilon_{k-1})},$$
which is equivalent to 
$$R^{1+\tilde{w}_j}\frac{\rho_0^{(1)}}{\rho_0^{(j)}}\leq R^{\tau(w_j-w_1)\frac{\tau w_1-\tilde{w}_1}{1+\tilde{w_1}}(n_k-2\log_{R}\varepsilon_{k-1})}.$$
When $w_j\neq w_1$, the above inequality follows by \eqref{eq:nk}, since $n_k\geq \frac{2(1+\tilde{w}_j)(1+\tilde{w_1})}{(w_j-w_1)(\tau w_1-\tilde{w}_1)}$.\\
{If $w_1=w_j$, then $\tilde{w}_1=\tilde{w}_j$, which implies that $n_k^{(1)}=n_k^{(j)}$ and this case does not occur. }
\end{proof}

\begin{lemma}
\label{lem:newlem}
There exist constants $\varepsilon_0$ and $R_0$ only depending on $\delta, \tau,\bs w$ etc. such that for every $\varepsilon<\varepsilon_0$ and every $R\geq R_0$ it holds that
$$\rho^{-({\xi} n_k^{(d)}-n_k^{(d)}-1)(1-\varepsilon)}\leq 2^{-(3d+1)}\rho_0^{(1)}\cdot R^{(\tau w_1-\tilde w_1)n_k^{(1)}}\cdot R^{-(d+1)({\xi} n_k^{(d)}-n_k^{(d)}-1)}.$$
Moreover, for any fixed $\varepsilon$ as above, there exists $R(\varepsilon)$ such that whenever $R\geq R(\varepsilon)$ it holds that
$$\frac{1}{2}\geq \rho^{-({\xi} n_k^{(d)}-n_k^{(d)}-1)\varepsilon/2}$$
for all $k\in\mb N$.
\end{lemma}

\begin{proof}
We aim to prove that
$$\rho^{-\alpha_k(1-\varepsilon)}\leq 2^{-(3d+1)}\rho_0^{(1)}\cdot R^{(\tau w_1-\tilde w_1)n_k^{(1)}}\cdot R^{-(d+1)\alpha_k},$$
where
$$\alpha_k={\xi} n_k^{(d)}-n_k^{(d)}-1.$$
This can be written as
$$2^{3d+1}\cdot R^{-(\tau w_1-\tilde w_1)n_k^{(1)}}\cdot R^{(d+1)\alpha_k}\leq \rho^{\alpha_k(1-\varepsilon)}\cdot \rho_0^{(1)}.$$
By using the bounds
$$\frac{R^{d+1}}{2^d}\leq \rho\leq R^{d+1},$$
we can rewrite this as
$$(\tau w_1-\tilde w_1)n_k^{(1)}-\frac{(d\alpha_k+3d+1)\log 2}{\log R}\geq \varepsilon \cdot\alpha_{k}(d+1)-\frac{\log \rho_0^{(1)}}{\log R}.$$
Now, suppose that
\begin{equation}
\label{eq:Rbig}
\frac{(d\alpha_k+3d+1)\log 2}{\log R}+\frac{|\log \rho_0^{(1)}|}{\log R}\leq \frac{(\tau w_i-\tilde w_1)n_k^{(1)}}{2}.    
\end{equation}
Then we just need to choose $\varepsilon$ small enough so that
$$\frac{(\tau w_i-\tilde w_1)n_k^{(1)}}{2}\geq \varepsilon\alpha_k(d+1)$$
and this can easily be done, by \eqref{eq:nkdup}. To ensure \eqref{eq:Rbig}, it is enough to choose $R$ large enough in terms of the constants. The second inequality in the statement is obvious.
\end{proof}

\subsection{Properties of the Cantor Set} 
In order to define the set $E_{\infty}$, we shall first construct collections $\mathcal{E}_l$ for $l\in \N\cup \{0\}$ consisting of disjoint boxes such that the sequence $\{\mathcal{E}_l\}_{l\in \N\cup\{0\}}$ is nested. By this, we mean that for each $l\in \N$ and for each box $E\in \mathcal{E}_{l}$, there exists a unique box $E'\in \mathcal{E}_{l-1}$ such that $E\subset E'$. Additionally, the collection $\mathcal{E}_{l}$ will be constructed satisfying the following properties:\\\\
\textbf{Dimensional properties}:
\begin{enumerate}[label=(D\arabic*), ref=D\arabic*]
  \item \label{D1} For $k\in \N$ and $n_k^{(d)}<\ell\leq n_{k+1}$ every box $E\in \mathcal{E}_l$ has side length $$\rho_0^{(i)}R^{-(1+\tilde{w_i})l}.$$
  \item \label{D2} For $k\in \N$ and $n_k<\ell\leq n_k^{(d)}$ every $E\in \mathcal{E}_l$ has side length
$$\rho_0^{(i)}R^{-(1+\tilde{w_i})\min\{l,n_{k}^{(i)}\}}.$$
        \end{enumerate}
\textbf{Pointwise properties}:\\
For each box $E\in \mathcal{E}_l$ and for every $\bs x\in E$ it holds that
\begin{enumerate}[label=(P\arabic*), ref=P\arabic*]
  \item \label{P1} If $l=n_k^{(d)}$ for some $k\in \N$, then there exists a rational vector 
        $\bs p_k/q_k$ (depending on $\bs x$) such that
        \begin{equation}
        \label{eq:p11}
            R^{n_k}\leq q_k \leq \frac{1}{\varepsilon_{k-1}}R^{n_k},
        \end{equation}
     \begin{equation}
     \label{eq:p12}
           (2c_k-1)^{1/w_1}\frac{1}{q_k^{\tau}}\leq \|q_k \bs x-\bs p_k\|_{\bs w}<\frac{1}{q_k^{\tau}}, 
        \end{equation}
        and
        $$\|s\bs x-\bs r\|_{\bs w}>\frac{1}{s^{\tau}} \quad \text{ for every }\frac{\bs r}{s}\neq \frac{\bs p_k}{q_k} \text{ with } R^{n_k}\leq s<R^{n_k^{(1)}}.$$
  
  \item \label{P2}
  If $n_k^{(d)}+1<l\leq n_{k+1}$ for some $k\in \N$, then  $$\|s\bs x-\bs r\|_{\tilde{\bs w}}> \frac{\varepsilon_0}{s} \quad\text{ for every }\frac{\bs r}{s} \text{ with } R^{n_{k}^{(d)}+1}\leq s<R^{l}.$$
  
  \item \label{P3} If $l={\xi} n_k^{(d)}$ for some $k\in \N$, then 
  $$\|s\bs x-\bs r\|_{\bs w}>\frac{1}{s^{\tau}} \quad\text{ for every }\frac{\bs r}{s} \text{ with } R^{n_k^{(1)}}\leq s<R^{n_k^{(d)}+1}.$$
  \end{enumerate}

Assume now that collections $\mc E_{\ell}$ satisfying Properties \eqref{D1}, \eqref{D2}, and \eqref{P1}-\eqref{P3} exist, and consider the Cantor set 
$$E_{\infty}:=\bigcap_{l=0}^{\infty}\bigcup_{E\in \mathcal{E}_l}E.$$
The following proposition shows that Properties \eqref{P1}-\eqref{P3} are enough to ensure that $E_{\infty}\subseteq E_{\bs w}(\tau)$. Properties \eqref{D1} and \eqref{D2} will be key in estimating the Hasudorff dimension of the set $E_{\infty}$.

\begin{proposition}
For the set $E_{\infty}$ defined above it holds that
$$E_{\infty}\subseteq E_{w}(\tau).$$
\end{proposition}
\begin{proof}
Let $\bs x\in E_{\infty}$. Since $\bs x\in \mathcal{E}_{n_k^{(d)}}$ for every $k\in \N$, by \eqref{P1}, there exists a sequence $\bs p_k/q_k$ such that $q_k\rightarrow \infty$ and $\|q_k\bs x-\bs p_k\|_{\bs w}\leq q_k^{-\tau}$. This implies that $\bs x\in W_{\bs w}(\tau)$. Let $c<1$. Since $(2c_k-1)^{1/w_1}\rightarrow 1$ as $k\rightarrow \infty$, there exists $k_0$ such that for every $k\geq k_0$, $c<(2c_k-1)^{1/w_1}$. Hence, by \eqref{eq:p12}, $cq_k^{-\tau}\leq \|q_k \bs x-\bs p_k\|_{\bs w}$. By \eqref{P2} and \eqref{P3}, it follows that $\|q\bs x-\bs p\|_{\bs w}\geq c{q^{-\tau}}$ for every $q\geq q_{k_0}$. Thus $\bs x\notin W_{w}(c,\tau)$ for every $c<1$. This gives $\bs x\in E_{\bs w}(\tau)$. 
\end{proof}

We conclude this subsection by proving the following useful lemma.
 
\begin{lemma}
\label{lem:nkbad}
   Let $k\in\mb N$ and assume that collections $\mc E_{\ell}$ satisfying properties $\eqref{P1}-\eqref{P3}$ exist for $\ell\leq n_{k+1}$. Then for every $E\in \mathcal{E}_{n_{k+1}}$, and $\bs x\in E$ we have that 
   \begin{equation*}
    \|s\bs x-\bs r\|_{\tilde{\bs w}}>\frac{\varepsilon_{k}}{s} \text{ for every } \frac{\bs r}{s}\in\mb Q^d \text{ such that } R^{n_k}\leq s< R^{n_{k+1}}.
\end{equation*}
\end{lemma}
\begin{proof}
Let $\bs x\in E$ and $\bs r/s$ be such $R^{n_k}\leq s<R^{n^{(d)}_k+1}$ and $\bs r/s\neq \bs p_k/q_k$. Then by \eqref{P1} and \eqref{P3}, there exists $1\leq i\leq d$ such that
   
$$|sx_i-r_i|>\frac{1}{s^{\tau w_i}}>\frac{1}{R^{\tau w_i (n^{(d)}_k+1)}}\overset{\eqref{eq:epsk}}{\geq }\frac{\varepsilon_k^{\tilde{w}_i}}{R^{n_k\tilde{w}_i}}\geq \frac{\varepsilon_k^{\tilde{w}_i}}{s^{\tilde{w}_i}}.$$
If $\bs r/s=\bs p_k/q_k$, then by \eqref{P1},
$$|sx_i-r_i|>\frac{(2c_k-1)^{w_i/w_1}}{s^{\tau w_i}}>\frac{1}{R^{\tau w_i (n^{(d)}_k+1)}}\overset{\eqref{eq:epsk}}{\geq }\frac{\varepsilon_k^{\tilde{w}_i}}{R^{n_k\tilde{w}_i}}\geq \frac{\varepsilon_k^{\tilde{w}_i}}{s^{\tilde{w}_i}}.$$
If $\bs r/s$ is such that $R^{n_k^{(d)}+1}\leq s<R^{n_{k+1}}$, then by \eqref{P2}
$$\|s\bs x-\bs r\|_{\tilde{\bs w}}>\frac{\varepsilon_0}{s}>\frac{\varepsilon_k}{s}.$$
\end{proof}

\subsection{Construction of the Collections $\mathcal{E}_l$}
\label{sec:3.2}
We now construct the collections $\mc E_{\ell}$ so that Properties \eqref{D1}, \eqref{D2}, \eqref{P1}-\eqref{P3} are satisfied.\\\\
Starting from a fixed corner, subdivide $[0,1)^d$ into boxes of side length $$\rho_0^{(i)}=\frac{R^{-(1+\tilde{w}_i)}}{(d+1)!}.$$
Denote by $\mathcal{E}_0$ the collection of these boxes.
Then 
$$ \# \mathcal{E}_0 =\floor*{\frac{1}{\rho_0^{(1)}}}\times \dots \times \floor*{\frac{1}{\rho_0^{(d)}}}.$$
Assuming that $\mathcal{E}_{n}$ is constructed for $0\leq n\leq l-1$, we shall now construct $\mathcal{E}_{l}$ by defining the set
$$\mathcal{E}_l(E)=\{E'\in\mc E_{\ell}:E'\subset E\}$$
for every $E\in \mathcal{E}_{l-1}$, based on the value of $\ell$ within the range $n_k<\ell\leq n_{k+1}$.

\bigskip
 
 \noindent \textbf{Case 1}: ${\xi} n_{k}^{(d)}<\ell\leq n_{k+1}$ for some $k\in \N$.

\bigskip

By \eqref{D1}, $E$ has side length $\rho_0^{(i)} R^{-(1+\tilde{w}_i)(l-1)}$.
Denote by $\mathcal{F}_l(E)$, the collection of boxes obtained by subdividing $E$ from a fixed corner into boxes of side length $\rho_0^{(i)}R^{-(1+\tilde{w}_i)l}$. Then
$$\# \mathcal{F}_l(E)= \prod_{i=1}^d \floor*{R^{(1+\tilde{w_i})}}=: \rho.$$
By applying Lemma \ref{lem:hyp} for the box $E$ and $\bs u=\tilde{\bs w}$, $n=l-1$, and $\varepsilon=\varepsilon_0$ (note that $\varepsilon_0$ satisfies the hypotheses of Lemma \ref{lem:hyp}), we deduce that the rational points  
$$\left\{\frac{\bs r}{s}\in \Q^{d}: R^{l-1}\leq s<R^{l} \text{ and } \|s\bs x-\bs r\|_{\tilde{w}}\leq \frac{\varepsilon_0}{s} \text{ for some } \bs x\in E\right\}$$ 
lie on an affine hyperplane in $\R^{d}$, which we denote by $\pi_E$. Consider the following thickening of this plane
$$N(\pi_E):=\left\{\bs x\in E: \exists\, \bs z\in \pi_{E} \text{ such that }|x_i-z_i|\leq \frac{\varepsilon_0^{\tilde{w}_i}}{R^{(l-1)(1+\tilde{w}_i)}}\ \mbox{for }i=1,\dotsc,d \right\}.$$
Let
$$\mathcal{B}_{l}(E):=\{F\in \mathcal{F}_{l}(E): F\cap N(\pi_E)\neq \emptyset\}$$
and define
\begin{equation}
\label{eq:elcase1}
\mathcal{E}_{l}(E):=\mathcal{F}_{l}(E)\setminus \mathcal{B}_{l}(E).    
\end{equation}

\begin{lemma}
\label{lem:case1}
For any $E\in\mc E_{\ell-1}$ and $\mc E_{\ell}(E)$ as in \eqref{eq:elcase1} we have that
$$\#\mathcal{E}_l(E)\geq \rho(1-2^d \floor{\rho_1}^{-1}).$$
Moreover, for every box $E'\in \mc E_{\ell}(E)$ Properties \eqref{D1} and \eqref{P2} hold.    
\end{lemma}

\begin{proof}
For every $i=1,\dotsc,d$ the thickness of $N(\pi_E)$ in the $i$-th  direction is $$ \frac{\varepsilon_0^{\tilde{w}_i}}{R^{(l-1)(1+\tilde{w}_i)}}\leq \rho_0^{(i)}R^{-(1+\tilde{w_i})l},$$
where the inequality follows from the definition of $\varepsilon_0$.
Hence, $\#\mathcal{B}_l(E)\leq 2^d \prod_{i=2}^d \floor{R^{1+\tilde{w_i}}}=2^d \rho \floor{R^{1+\tilde{w_1}}}^{-1}$.
This implies that
$$\#\mathcal{E}_l(E)\geq \rho(1-2^d \floor{\rho_1}^{-1}).$$
Property \eqref{D1} obviously holds. Then, by the inductive hypothesis, we are left to show that for every box $E'\in \mathcal{E}_l(E)$ and $\bs x\in E'$
$$\|s\bs x-\bs r\|_{\tilde{w}}>\frac{\varepsilon_0}{s} \quad\text{ for every }\frac{\bs r}{s} \text{ with } R^{l-1}\leq s<R^l.$$
If not, there exists $\bs x\in E'$ and $\bs r/s\in \Q^d$ such that $R^{l-1}\leq s<R^{l}$ and $\|s\bs x-\bs r\|_{\tilde{\bs w}}\leq \varepsilon_0/s$. Since $E'\subseteq E$, we have $\bs r/s\in \pi_E$. By construction of $\mathcal{E}_l(E)$, $E'\cap N(\pi_E)=\emptyset$. Hence, $\bs x\notin N(\pi_E)$. Therefore, there exists $1\leq i\leq d$ such that 
$$\left|x_i-\frac{r_i}{s}\right|>\frac{\varepsilon_0^{\tilde{w}_i}}{R^{(l-1)(1+\tilde{w}_i)}},$$
which implies that $\|sx-r\|_{\tilde{\bs w}}>\varepsilon_0/s$ --- a contradiction.
\end{proof}

\bigskip

\noindent \textbf{Case 2}: $n_k<l\leq n_k^{(d)}$ for some $k\in \N$.

\bigskip
 
Let us assume first that $\ell=n_k+1$. Let $\bs z$ the center of $E$. By Lemma \ref{lem:nkbad} and the inductive hypothesis, we have that
\begin{equation*}
          \|s\bs z-\bs r\|_{\tilde{\bs w} }> \frac{\varepsilon_{k-1}}{s} \; \text{ for every } \frac{\bs r}{s}\in \Q^d \text{ with }\; 1\leq s< R^{n_k}.
\end{equation*}
Applying Lemma \ref{lem:pk/qk} with $\varepsilon=\varepsilon_{k-1}$, $M=R^{n_k}$, $u=\tilde{\bs w}$, $\bs x=\bs z$, and $\beta>\varepsilon_{k-1}^{-1}$, we find a rational vector
$$\frac{\bs p_k}{q_k}=\left(\frac{p_{k1}}{q_k},\dots, \frac{p_{kd}}{q_k}\right)\in \Q^d$$
such that 
\begin{equation}\label{eq:200}
R^{n_k}\leq q_k\leq \frac{1}{\varepsilon_{k-1}}R^{n_k}
\end{equation}
      and 
      \begin{equation}
      \label{eq:onequarter}
          \left|z_i-\frac{p_{ki}}{q_k}\right|\leq \frac{\varepsilon_{k-1}^{\tilde{w}_i}}{R^{(1+\tilde{w}_i) n_k}}\leq \frac{1}{4}\rho_0^{(i)}R^{-(1+\tilde{w}_i)n_k}  \quad \mbox{for } i=1,\dotsc,d,
      \end{equation}
where the second inequality follows from the definition of $\varepsilon_0$. Now note that, by \eqref{eq:200},
\begin{equation}
\label{eq:anotherquarter}
q_k^{1+\tau w_i}\geq R^{(1+\tau w_i)n_k}\geq  \frac{4 R^{(1+\tilde w_i)n_k}}{\rho_0^{(i)}},    
\end{equation}
where the last inequality follows form \eqref{eq:nk} and the fact that $\varepsilon_{k-1}^{-1}>\varepsilon_0^{-1}>4(\rho_0^{(i)})^{-1}$. Since the side length of $E$ is 
$\rho_0^{(i)}R^{-(1+\tilde w_i)n_k}$, by \eqref{eq:onequarter} and \eqref{eq:anotherquarter}, there exists $\bs y\in E$ such that
\begin{equation}\label{eq:100}
\left|y_i-\frac{p_k^{(i)}}{q_k}\right|=c_k \frac{1}{q_k^{1+\tau w_i}} \quad \mbox{for } i=1,\dotsc,d.
\end{equation}
Now, given that $\rho_0^{(i)}R^{-(1+\tilde{w}_i)n_k}>\rho_0^{(i)}R^{-(1+\tilde{w}_i)\min\{l, n_k^{(i)}\}}$, we may choose a box of side length $\rho_0^{(i)}R^{-(1+\tilde{w}_i)\min\{l, n_k^{(i)}\}}$ inside $E$ containing $\bs y$. We denote this box by $E_l(\bs y)$. Then we put 
\begin{equation}
\label{eq:elcase2}
\mathcal{E}_l(E):=\{E_l(\bs y)\}.    
\end{equation}
For $\ell>n_{k}+1$, we denote by $E_{\ell}(\bs y)$ any box of side lengths $\rho_0^{(i)}R^{-(1+\tilde{w}_i)\min\{l, n_k^{(i)}\}}$ such that $y\in E_{\ell}(\bs y)\subset E_{\ell-1}(\bs y)$. Again, we define $\mc E_{\ell}(E)$ as in \eqref{eq:elcase2}.

\begin{lemma}\label{lem:case2}
For $\mc E_{\ell}(E)$ as in \eqref{eq:elcase2}, we have that
$$\# \mc E_{\ell}(E)=1.$$
Moreover, for any $E'\in\mc E_{\ell}(E)$ Properties \eqref{D2} and \eqref{P1} hold.
\end{lemma}

\begin{proof}
The only non-trivial part of the statement is the validity of \eqref{P1}. Let $\bs x\in E_{n_k^{(d)}}(\bs y)$. By construction, $y\in E_{n_k^{(d)}}(\bs y)$ and $E_{n_k^{(d)}}(\bs y)$ has side lengths
$$\rho_0^{(i)}R^{-(1+\tilde{w}_i)\min\{n_k^{(d)}, n_k^{(i)}\}}=\rho_0^{(i)}R^{-(1+\tilde{w}_i)n_k^{(i)}}.$$
Hence, for every $i=1,\dotsc,d$
\begin{equation}\label{eq:101}
|x_i-y_i|\leq \rho_{0}^{(i)}R^{-(1+\tilde{w}_i)n_k^{(i)}}\overset{\eqref{eq:ckcond}}{\leq} (1-c_k)\varepsilon_{k-1}^{1+\tau w_i}R^{-(1+\tau w_i)n_k} \overset{\eqref{eq:200}}{\leq}(1-c_k)\frac{1}{q_k^{1+\tau w_i}}.
\end{equation}
By \eqref{eq:100} and \eqref{eq:101}, we have that
\begin{equation}\label{eq: 102}
    \left|x_i-\frac{p_{ki}}{q_k}\right|\leq |x_i-y_i|+ \left|y_i-\frac{p_{ki}}{q_k}\right|\leq \frac{1}{q_k^{1+\tau w_i}}
\end{equation}
and 
\begin{equation}
    \left|x_i-\frac{p_{ki}}{q_k}\right|\geq \left|y_i-\frac{p_{ki}}{q_k}\right|-|x_i-y_i|\geq (2c_k-1)\frac{1}{q_k^{1+\tau w_i}}.
\end{equation}
Hence, 
\begin{equation*}
   (2c_k-1)^{1/w_1}\frac{1}{q_k^{\tau}}\leq  \|q_k\bs x-p_k\|_{\bs w}\leq \frac{1}{q_k^{\tau}},
\end{equation*}
which gives \eqref{eq:p11} and \eqref{eq:p12}. Now, for fixed $\bs x\in E_{n_k^{(d)}}(\bs y)$ define the unimodular lattice $\Lambda_{\bs x}$ in $\R^{d+1}$ as 
\[
\Lambda_{\bs x} \;:=\;
\begin{pmatrix}
    \textup{I}_{d \times d} & \bs x_{d \times 1} \\[6pt]
    \bs{0}_{1\times d}              & 1
\end{pmatrix}
\mathbb{Z}^{\,d+1}.
\]
Consider the convex set
$$K_{k}:=\prod_{i=1}^d\left[-\frac{\varepsilon_{k-1}^{\tilde{w}_i}}{R^{n_k \tilde{w}_i}}, \frac{\varepsilon_{k-1}^{\tilde{w}_i}}{R^{n_k \tilde{w}_i}}\right]\times \left[-R^{n_k},R^{n_k}\right].$$
By Lemma \ref{lem:nkbad} and the inductive hypothesis,
$$\|s\bs x-\bs r\|_{\tilde{w}} > \frac{\varepsilon_{k-1}}{s}\mbox{ for every rational }\frac{\bs r}{s}\mbox{ with } 1\leq s\leq R^{n_k}.$$
Hence, $K_k$ does not contain any non-zero vector of $\Lambda_{\bs x}$. Therefore $\lambda_1(K_k,\Lambda_{\bs x})> 1$ and, by the upper bound in \eqref{eq:mink}, we deduce that $$\lambda_{d+1}(K_k,\Lambda_{\bs x})< \frac{2^{d+1}}{\vol(K_k)}=\frac{1}{\varepsilon_{k-1}}.$$
Hence, $\frac{1}{\varepsilon_{k-1}}K_k$ contains $d+1$ linearly independent vectors of the lattice $\Lambda_{\bs x}$.
Define
\begin{equation*}
    {\tilde K_k}:=\prod_{i=1}^{d}\left[-\frac{R^{-(\tilde{w}_i-\tilde{w}_1)n_k}}{R^{\tau w_1 n_k}}, \frac{R^{-(\tilde{w}_i-\tilde{w}_1)n_k}}{R^{\tau w_1 n_k}}\right] \times \left[-R^{(1+\tau w_1-\tilde{w}_1)n_k}\varepsilon_{k-1}^{d},R^{(1+\tau w_1-\tilde{w}_1)n_k}\varepsilon_{k-1}^d\right]
\end{equation*}
From \eqref{eq:nk}, we have that $R^{2(\tau w_1-\tilde w_1)n_k}\geq \varepsilon_{k-1}^{-d}$. Therefore,
$$\frac{1}{\varepsilon_{k-1}}K_k\subseteq \frac{R^{(\tau w_1-\tilde{w}_1)n_k}}{\varepsilon_{k-1}}{\tilde K_k}.$$
Hence, the convex set $\frac{R^{(\tau w_1-\tilde{w}_1)n_k}}{\varepsilon_{k-1}}{\tilde K_k}$ contains $d+1$ linearly independent vectors of the lattice $\Lambda_{\bs x}$. This implies that
\begin{equation}
\label{eq:lastmin}
\lambda_{d+1}({\tilde K_k},\Lambda_{\bs x})\leq \frac{R^{(\tau w_1-\tilde{w}_1)n_k}}{\varepsilon_{k-1}}.    
\end{equation}
Now, observe that $\tau w_1-\tilde w_1\leq \tau w_i-\tilde w_i$ for $i\geq 1$.
This follows from the fact that $(1+\tau w_1)/(1+\tilde{w}_1)\leq (1+\tau w_i)/(1+\tilde{w}_i)$. Hence,
\begin{equation}
\label{eq:neww}
\frac{1}{R^{\tau w_i n_k}}\leq \frac{R^{-(\tilde{w}_i-\tilde{w}_1)n_k}}{R^{\tau w_1 n_k}}.    
\end{equation}
By \eqref{eq:200}, \eqref{eq: 102}, \eqref{eq:neww}, and the fact that $R^{(\tau w_1-\tilde w_1)n_k}\geq \varepsilon_{k-1}^{-(d+1)}$, we deduce that $(q_k\bs x-\bs p_k,q_k)\in \Lambda_{\bs x}\cap {\tilde K_k}$ and thus
\begin{equation}
\label{eq:firstmin}
\lambda_1({\tilde K_k},\Lambda_{\bs x})\leq 1.    
\end{equation}
By the lower bound in \eqref{eq:mink}, we see that 
$$\lambda_1({\tilde K_k},\Lambda_{\bs x})\, \lambda_2({\tilde K_k},\Lambda_{\bs x})\,\lambda_{d+1}({\tilde K_k},\Lambda_{\bs x})^{d-1}\geq \frac{2^{(d+1)}}{(d+1)!\,\vol(\tilde K_k)}$$
Combining this with \eqref{eq:lastmin} and \eqref{eq:firstmin}, we find
$$\lambda_2({\tilde K_k},\Lambda_{\bs x}) \geq \frac{2^{d+1}\varepsilon_{k-1}^{d-1}}{(d+1)!\,\vol({\tilde K_k})\, R^{(\tau w_1-\tilde{w}_1)(d-1)n_k}}.$$
Now, note that
$$\vol({\tilde K_k})=2^{d+1}R^{-(d-1)(\tau w_1-\tilde{w}_1)n_k}\varepsilon_{k-1}^d.$$ Then
$$\lambda_2({\tilde K_k},\Lambda_{\bs x}) \geq \frac{1}{\varepsilon_{k-1}(d+1)!}>1,$$
which follows from the definition of $\varepsilon_0$.
Hence, for every $\bs r/s\neq \bs p_k/q_k$, with $s\leq R^{(1+\tau w_1-\tilde{w}_1)n_k}\varepsilon_{k-1}^d$, there exists $1\leq i \leq d$ such that
$$|sx_i-r_i|>\frac{R^{-(\tilde{w}_i-\tilde{w}_1)n_k}}{R^{\tau w_1 n_k}}\stackrel{\eqref{eq:neww}}{\geq} \frac{1}{R^{\tau w_i n_k}}.$$
If $s\geq R^{n_k}$, then $ R^{-\tau w_i n_k}\geq s^{-\tau w_i}$. This shows that $\eqref{P1}$ holds since, by \eqref{eq:wave}, $R^{n_k^{(1)}}\leq R^{(1+\tau w_1-\tilde{w}_1)n_k}\varepsilon_{k-1}^d$.
\end{proof}

\bigskip

\noindent\textbf{Case 3}: $l= n_{k}^{(d)}+1$ \text{ for some } $k\in \N$

\bigskip

Since $l-1= n_k^{(d)}$, by the inductive hypothesis, $E$ has 
side length
$\rho_{0}^{(i)}R^{-(1+\tilde{w}_i)n_k^{(i)}}$. We shall now define a region $A_k(E)\subseteq E$ which needs to be removed systematically in the subsequent steps to ensure that \eqref{P3} holds. Once $A_k(E)$ is constructed, we will proceed to define $\mc E_{\ell}(E)$.\\\\
\textbf{Construction of $A_k(E)$}: For $n_k^{(1)}\leq n\leq n_k^{(d)}+1$, define $\mathcal{I}_n(E)$ to be the collection of boxes obtained by subdividing $E$ from a fixed corner into boxes of side length
$$\rho_{0}^{(i)}R^{-(1+\tilde{w}_i)\max\{n,n_k^{(i)}\}}$$
and denote by $\mathcal{R}_n(E)$, the collection of remainder boxes at each level, i.e., those boxes (of potentially smaller size) that are left after $E$ is subdivided into an integer number of boxes of the required side length. Note that this time we are not iteratively subdividing: for each $n$ we subdivide the initial box $E$ and, in particular, we have that
$$E=\bigcup_{I\in\mc I_n(E)\cup\mc R_n(E)}I.$$
Then the side length of any box in $\mathcal{R}_n(E)$ is $< \rho_{0}^{(i)}R^{-(1+\tilde{w}_i)\max\{n,n_k^{(i)}\}}$ and 
\begin{align}\label{eq:card}
\#\bigl(\mathcal{I}_n(E)\cup \mathcal{R}_n(E)\bigr)
&\le \prod_{i=1}^d \left(
  \floor*{R^{(1+\tilde w_i)\,(\max\{n,n_k^{(i)}\}-n_k^{(i)})}} + 1
\right) \notag\\
&\le 2^d \prod_{i=1}^d
   R^{(1+\tilde w_i)\,(\max\{n,n_k^{(i)}\}-n_k^{(i)})}.
\end{align}
By Lemma \ref{lem:hyp} applied to each box $I\in \mathcal{I}_n(E) \cup \mathcal{R}_n(E)$ with $\varepsilon=R^{-n_k(\tau-1)}<\varepsilon_0$, the points
$$\left\{\frac{\bs r}{s}\in\Q^{d}: R^n\leq s<R^{n+1} \text{ and }\|s\bs x-\bs r\|_{\bs w}\leq \frac{1}{s^{\tau}} \text{ for some }\bs x\in I\right\}$$
lie on an affine hyperplane in $\R^d$, which we denote by $\pi_I$.
Now, consider the neighborhood of this plane given by
\begin{equation}
\label{eq:defnpi}
N_{\tau}(\pi_I):=\left\{\bs x\in I: \exists\, \bs z\in \pi_{I} \text{ such that }|x_i-z_i|\leq \frac{1}{R^{(1+\tau w_i)n}} \mbox{ for all } i=1,\dotsc,d\right\}.    
\end{equation}
Then
$$\vol (N_{\tau}(\pi_I))\leq \max_{j=1}^d\left\{R^{-(1+\tau w_j)n}\prod_{i\neq j}^d \rho_0^{(i)}R^{-(1+\tilde{w}_i)\max\{n,n_k^{(i)}\}}\right\}.$$
By \eqref{eq:volnpi} in Lemma \ref{lem:lemma5}, we have that
\begin{multline}
\label{eq:hypvol}
\max_{j=1}^d\left\{R^{-(1+\tau w_j)n}\prod_{i\neq j}^d \rho_0^{(i)}R^{-(1+\tilde{w}_i)\max\{n,n_k^{(i)}\}}\right\} \\
=R^{-(1+\tau w_1)n}\prod_{i=2}^d \rho_0^{(i)}R^{-(1+\tilde{w}_i)\max\{n,n_k^{(i)}\}}.    
\end{multline}
Thus, \eqref{eq:defnpi} and \eqref{eq:hypvol} imply that
 \begin{align}\label{eq:volntau}
     \vol (N_{\tau}(\pi_I))\leq &R^{-(1+\tau w_1)n}\prod_{i=2}^d \rho_0^{(i)}R^{-(1+\tilde{w}_i)\max\{n,n_k^{(i)}\}} \notag\\
     & =\frac{R^{-(\tau w_1-\tilde{w}_1)n}}{\rho_0^{(1)}}\prod_{i=1}^d \rho_0^{(i)}R^{-(1+\tilde{w}_i)\max\{n,n_k^{(i)}\}}.
 \end{align}
Define $$A_k^{(n)}(E):=\bigcup_{I\in \mathcal{I}_n(E)\cup \mathcal{R}_n(E)}N_{\tau}(\pi_I) \subseteq E$$
Let $$A_k(E):=\bigcup_{n=n_k^{(1)}}^{n_k^{(d)}+1} A_k^{(n)}(E) \quad \text{and}\quad  A_k:=\bigcup_{E\in \mathcal{E}_{n_k^{(d)}}}A_k(E)$$
By construction, if $\bs x\notin A_k$ then 
\begin{equation}
\label{eq:toproveP3}
\|s\bs x-\bs r\|_{\bs w}>\frac{1}{s^{\tau}} \quad \text{ for every }\frac{\bs r}{s} \text{ with } R^{n_k^{(1)}}\leq s<R^{n_k^{(d)}+2}.    
\end{equation}
To see this, just note that if $\|s\bs x-\bs r\|_{\bs w}\leq s^{-\tau}$ for some $R^{n}\leq s <R^{n+1}$ with $n_k^{(1)}\leq n <k^{(d)}+1$, since $\bs x\in I$ for some $\mc I_n$, it must be $\bs r/s\in\pi_I$ so that $\bs x\in N_{\tau}(\pi_I)$ --- a contradiction.

The region $A_k$ is a union of small trapezoids inside $E$. We shall now construct a collection $\mathcal{A}_k$ of boxes which are ``small enough" to fit inside these trapezoids.\\
Denote by $\mathcal{J}_{{\xi} n_k^{(d)}}(E)$, the collection obtained by first subdividing $E$ into boxes of side length $\rho_0^{(i)}R^{-(1+\tilde{w}_i)(n_k^{(d)}+1)}$ and then iteratively subdividing each of these boxes into boxes of side length $\rho_0^{(i)}R^{-(1+\tilde{w}_i)n}$ until we reach side length $\rho_0^{(i)}R^{-(1+\tilde{w}_i){\xi} n_k^{(d)}}$. Then
$$\# \mathcal{J}_{{\xi} n_k^{(d)}}(E)= \rho^{{\xi} n_k^{(d)}-n_k^{(d)}-1}\prod_{i=1}^d \floor*{R^{(1+\tilde{w}_i)(n_k^{(d)}+1-n_k^{(i)})}},$$
where $\rho=\prod_{i=1}^d \floor{R^{1+\tilde{w}_i}}$.
Define
$$\mathcal{A}_k(E):=\{J\in \mathcal{J}_{{\xi} n_k^{(d)}}(E): J \cap A_k(E)\neq \emptyset \}$$
and let $$\mathcal{A}_k:=\bigcup_{E\in \mathcal{E}_{n_k^{(d)}}}\mathcal{A}_k(E).$$

\noindent
\textbf{Construction of $\mathcal{E}_{n_k^{(d)}+1}(E)$}: 
Subdivide $E$ into boxes of side length $\rho_0^{(i)}R^{-(1+\tilde{w}_i)(n_k^{(d)}+1)}$ and denote this collection by $\mathcal{F}_{n_k^{(d)}+1}(E)$. Then
\begin{equation}\label{eq:card1}
    \# \mathcal{F}_{n_k^{(d)}+1}(E)=\prod_{i=1}^{d}\floor*{R^{(1+\tilde{w}_i)(n_k^{(d)}+1-n_k^{(i)})}}\geq \frac{1}{2^d}R^{(d+1)(n_k^{(d)}+1)}\prod_{i=1}^d R^{-(1+\tilde{w}_i)n_k^{(i)}}.
\end{equation}
Fix $0<\varepsilon<\varepsilon_0$ and assume that $R\geq R(\varepsilon)$ as in Lemma \ref{lem:newlem}. Let
$$\mathcal{B}'_{n_k^{(d)}+1}(E):=\left\{F\in \mathcal{F}_{n_k^{(d)}+1}(E): \#\{J\in \mathcal{A}_k: J\subseteq F\}\geq \rho^{({\xi} n_k^{(d)}-n_k^{(d)}-1)(1-\varepsilon/2)} \right\},$$
where $\rho:=\prod_{i=1}^d\floor*{R^{1+\tilde{w}_i}}$, and define
\begin{equation}
\label{eq:elcase3}
\mathcal{E}_{n_k^{(d)}+1}(E)=\mathcal{F}_{n_k^{(d)}+1}(E)\setminus \mathcal{B}'_{n_k^{(d)}+1}(E).    
\end{equation}

\begin{lemma}\label{lem:case3}
For $\mathcal{E}_{n_k^{(d)}+1}(E)$ as in \eqref{eq:elcase3}, we have that
$$\# \mathcal{E}_{n_k^{(d)}+1}(E)\geq \frac{1}{2}\prod_{i=1}^d \floor*{R^{(1+\tilde{w_i})(n_k^{(d)}+1-n_k^{(i)})}}.$$
Moreover, for any $E'\in \mathcal{E}_{n_k^{(d)}+1}(E)$ Property \eqref{D1} holds, while the other properties are vacuous.    
\end{lemma}

\begin{proof}
The latter part of the statement is obvious. Thus, it is enough to prove the lower bound on the cardinality of $\# \mathcal{E}_{n_k^{(d)}+1}(E)$.

We shall first find an upper bound for $\# \mathcal{A}_k(E)$. Fix $n_k^{(1)}\leq n \leq n_k^{(d)}$ and $I\in \mathcal{I}_n(E)\cup \mathcal{R}_n(E)$. For every $1\leq i\leq d$ the side length of any box $J\in \mathcal{J}_{{\xi} n_k^{(d)}}(E)$ in the $i$-th direction is smaller than the thickness of the neighborhood $N_{\tau}(\pi_I)$ in that direction. This follows from the fact that
$$\rho_0^{(1)}R^{-(1+\tilde w_1)\xi n_k^{(d)}}<R^{-(1+\tau w_d)(n_k^{(d)}+1)}$$
(cf. Lemma \ref{lem:newlem1}). Then we have that
$$\#\{J\in \mathcal{J}_{{\xi} n_k^{(d)}}(E): J \cap N_{\tau}(\pi_I)\neq \emptyset \}\leq  2^d \frac{\vol (N_{\tau}(\pi_I))}{R^{-(d+1){\xi} n_k^{(d)}}\prod_{i=1}^d \rho_0^{(i)}}  $$
Hence,
\begin{align*}
    \#\{J\in \mathcal{J}_{{\xi} n_k^{(d)}}(E): & J \cap A_k^{(n)}(E)\neq \emptyset \} \leq  \frac{2^d}{R^{-(d+1){\xi} n_k^{(d)}}\prod_{i=1}^d \rho_0^{(i)}}\sum_{I\in \mathcal{I}_n(E)\cup \mathcal{R}_n(E)} \vol(N_{\tau}(\pi_I)) \\
    & \overset{\eqref{eq:card},\eqref{eq:volntau}}{\leq }  \frac{2^{2d} \prod_{i=1}^d R^{-(1+\tilde{w}_i)n_k^{(i)}}}{\rho_0^{(1)}R^{(\tau w_1-\tilde{w}_1)n} R^{-(d+1){\xi} n_k^{(d)}}}.
\end{align*}
Therefore,
\begin{align*}
    \#\{J\in \mathcal{J}_{{\xi} n_k^{(d)}}(E): J \cap A_k(E)\neq \emptyset \}
    &\leq \sum_{n=n_k^{(1)}}^{n_k^{(d)}+1} \frac{2^{2d} \prod_{i=1}^d R^{-(1+\tilde{w}_i)n_k^{(i)}}}{\rho_0^{(1)}R^{(\tau w_1-\tilde{w}_1)n} R^{-(d+1){\xi} n_k^{(d)}}} \\
    &\leq  \frac{2^{2d+1} \prod_{i=1}^d R^{-(1+\tilde{w}_i)n_k^{(i)}}}{\rho_0^{(1)}R^{(\tau w_1-\tilde{w}_1)n_k^{(1)}} R^{-(d+1){\xi} n_k^{(d)}}}\\
    &\leq  \# \mathcal{F}_{n_k^{(d)}+1}(E) \rho^{({\xi} n_k^{(d)}-n_k^{(d)}-1)(1-\varepsilon)},
\end{align*}
where in the last inequality we used the fact that, by Lemma \ref{lem:newlem},
$$\frac{2^{2d+1} \prod_{i=1}^d R^{-(1+\tilde{w}_i)n_k^{(i)}}}{\rho_0^{(1)}R^{(\tau w_1-\tilde{w}_1)n_k^{(1)}} R^{-(d+1){\xi} n_k^{(d)}}}\cdot \frac{R^{(d+1)(n_k^{(d)}+1)}}{R^{(d+1)(n_k^{(d)}+1)}}\overset{\eqref{eq:card1}}{\leq }\frac{\# \mathcal{F}_{n_k^{(d)}+1}(E)}{\rho^{-({\xi} n_k^{(d)}-n_k^{(d)}-1)(1-\varepsilon)}}.$$
This gives
\begin{equation}
\label{eq:34}
\# \mathcal{A}_k(E)\leq  \# \mathcal{F}_{n_k^{(d)}+1}(E) \rho^{({\xi} n_k^{(d)}-n_k^{(d)}-1)(1-\varepsilon)}.   
\end{equation}
Now, assume by contradiction that
\begin{equation}
\label{eq:contr}
\#\mathcal{B}'_{n_k^{(d)}+1}(E) > \frac{1}{2}\# \mathcal{F}_{n_k^{(d)}+1}(E).    
\end{equation}
Then
\begin{equation}\label{eq:cardak}
    \#\mathcal{A}_k(E)>\frac{1}{2}\# \mathcal{F}_{n_k^{(d)}+1}(E) \rho^{({\xi} n_k^{(d)}-n_k^{(d)}-1)(1-\varepsilon/2)}.
\end{equation}
Now, by Lemma \ref{lem:newlem}, we have that $\rho^{-({\xi} n_k^{(d)}-n_k^{(d)}-1)\varepsilon/2}\leq 1/2$. Thus, \eqref{eq:cardak} implies that
$$\#\mathcal{A}_k(E)>\frac{1}{2}\# \mathcal{F}_{n_k^{(d)}+1}(E) \rho^{({\xi} n_k^{(d)}-n_k^{(d)}-1)(1-\varepsilon/2)}\geq \# \mathcal{F}_{n_k^{(d)}+1}(E)\rho^{({\xi} n_k^{(d)}-n_k^{(d)}-1)(1-\varepsilon)}$$
--- a contradiction to \eqref{eq:34}. This means that the opposite of \eqref{eq:contr} holds and
\begin{multline*}
\#\mathcal{E}_{n_k^{(d)}+1}(E)=\# \mathcal{F}_{n_k^{(d)}+1}(E)-\#\mathcal{B}'_{n_k^{(d)}+1}(E)\geq \frac{1}{2}\# \mathcal{F}_{n_k^{(d)}+1}(E) \\
\overset{\eqref{eq:card1}}{=}\frac{1}{2}\prod_{i=1}^d \floor*{R^{(1+\tilde{w_i})(n_k^{(d)}+1-n_k^{(i)})}},    
\end{multline*}
concluding the proof.
\end{proof}

\bigskip

\noindent
\textbf{Case 4}: $n_{k}^{(d)}+2\leq \ell \leq {\xi} n_{k}^{(d)}$ for some $k\in \N$

\bigskip

By the inductive hypothesis, $E$ has side length $\rho_0^{(i)}R^{-(1+\tilde{w}_i)(l-1)}$. Denote by $\mathcal{F}_l(E)$, the collection of boxes obtained by subdividing $E$ from a fixed corner into boxes of side length $\rho_0^{(i)}R^{-(1+\tilde{w}_i)l}$. Then
$$\# \mathcal{F}_l(E)= \prod_{i=1}^d \floor*{R^{(1+\tilde{w_i})}}=\rho.$$
By Lemma \ref{lem:hyp} for the box $E$ with $\bs u=\tilde{\bs w}$, $n=l-1$, and $\varepsilon =\varepsilon_0$, we have that the points
$$\left\{\frac{\bs r}{s}\in \Q^{d}: R^{l-1}\leq s<R^{l} \quad\text{ and } \|s\bs x-\bs r\|_{\tilde{\bs w}}\leq \frac{\varepsilon_0}{s} \text{ for some } \bs x\in E\right\}$$ 
lie on an affine plane in $\R^{d}$, say $\pi_E$. Consider the neighborhood of this plane defined by 
$$N(\pi_E):=\left\{\bs x\in E: \exists\, \bs z\in \pi_{E} \text{ such that }|x_i-z_i|\leq \frac{\varepsilon_0^{\tilde{w}_i}}{R^{(l-1)(1+\tilde{w}_i)}} \mbox{ for } i=1,\dotsc,d\right\}.$$
Define
$$\mathcal{B}_{l}(E):=\{F\in \mathcal{F}_{l}(E): F\cap N(\pi_E)\neq \emptyset\}$$
and 
$$\mathcal{B}'_{l}(E):=\left\{F\in \mathcal{F}_{l}(E): \#\{J\in \mathcal{A}_k: J\subseteq F\}\geq \rho^{({\xi} n_k^{(d)}-l)(1-\varepsilon/2)} \right\}$$
and then define
\begin{equation}
\label{eq:elcase4}
\mathcal{E}_{l}(E):=\mathcal{F}_{l}(E)\setminus \left(\mathcal{B}_{l}(E)\cup \mathcal{B}'_{l}(E)\right).    
\end{equation}

\begin{lemma}\label{lem:case4}
For $\mc E_{\ell}(E)$ as in \eqref{eq:elcase4}, we have that
$$\# \mc E_{\ell}(E)\geq \rho(1-\rho^{-\varepsilon/2}-2^d \floor{\rho_1}^{-1}).$$
Moreover, for any $E'\in\mc E_{\ell}(E)$ Properties \eqref{D1}, \eqref{P2}, and \eqref{P3} hold.    
\end{lemma}

\begin{proof}
By the same argument as in the proof of Lemma \ref{lem:case1}, we have that
$$\# \mathcal{B}_l(E)\leq 2^d \prod_{i=2}^d \floor{\rho_i}.$$
Now, recall that for every box $E'$ in $\mathcal{B}'_{l}(E)$ it holds that
\begin{equation}\label{eq:card5}
    \#\{J\in \mathcal{A}_k : J\subseteq E'\}\geq \rho^{({\xi} n_k^{(d)}-l)(1-\varepsilon/2)}.
\end{equation}
Assume by contradiction that $\# \mathcal{B}_l'(E)\geq \rho^{1-\varepsilon/2}$. Then we have that
$$\#\{J\in \mathcal{A}_k : J\subseteq E\}\geq \rho^{(1-\varepsilon/2)}\rho^{({\xi} n_k^{(d)}-l)(1-\varepsilon/2)} = \rho^{({\xi} n_k^{(d)}-l+1)(1-\varepsilon/2)}.$$
This implies that $E\in \mc B_{\ell-1}'(E'')$, where $E''$ is the parent box of $E$. But this contradicts \eqref{eq:elcase4} for the box $E''$ at the level $\ell-1$ or \eqref{eq:elcase3} if $\ell-1=n_k^{(d)}+1$. Then
$$\# \mathcal{E}_{l}(E)\geq \rho-\rho^{1-\varepsilon/2}-2^d\prod_{i=2}^d\lfloor\rho_i\rfloor.$$

We are now left to verify Properties \eqref{P2} and \eqref{P3}. Property \eqref{P2} holds by the same argument as in the proof of lemma \ref{lem:case1}. For \eqref{P3}, we first observe that for the unique box $E_0\in \mathcal{E}_{n_k^{(d)}}$ such that $E\subseteq E_0$ it holds that
$$\mathcal{F}_{{\xi} n_k^{(d)}}(E)=\{E'\in \mathcal{J}_{{\xi} n_k^{(d)}}(E_0):E'\subseteq E\}.$$
Therefore,
\begin{multline*}
\mathcal{B}_{{\xi} n_k^{(d)}}'(E)=\left\{F\in \mathcal{F}_{{\xi} n_k^{(d)}}(E): \#\{J\in \mathcal{A}_k: J\subseteq F\}\geq \rho^{({\xi} n_k^{(d)}-{\xi} n_k^{(d)})(1-\varepsilon/2)}=1 \right\} \\
=\mathcal{A}_k\cap \mathcal{F}_{{\xi} n_k^{(d)}}(E).    
\end{multline*}
Then $\mathcal{A}_k \cap \mathcal{E}_{{\xi} n_k^{(d)}}(E)=\emptyset$ and the validity of $\eqref{P3}$ follows from \eqref{eq:toproveP3}.
\end{proof}

\section{Hausdorff Dimension of $E_{\infty}$}

We start by recalling \cite[Proposition 1.7]{Fal14}.

\begin{lemma}[Mass-Distribution Principle]
\label{lem:mdp}
Let $\mu$ be a compactly-supported probability measure on $\mb R^d$. Assume that there exist constants $0<\ell_0<1$ and $0\leq \beta\leq d$ such that for any $0<\ell\leq \ell_0$ and any hyper-cube $B\subset\mb R^d$ of side length $\ell$
$$\mu(B)\leq \ell^{\beta}.$$
Then any set $F\subset \mb R^d$ such that $\mu(F)>0$ satisfies
$$\dimH F\geq \beta.$$
\end{lemma}
\subsection{Counting}
In order to compute the Hausdorff dimension of the Cantor set $E_{\infty}$, we first need to give a lower bound on the number of boxes of $\mathcal{E}_l$ which are contained in a fixed box $E\in \mathcal{E}_{l-1}$. \\
Recall that $$\rho_i:=R^{1+\tilde{w_i}} \quad\text{ and }\quad \rho:=\prod_{i=1}^d \floor*{R^{1+\tilde{w_i}}}.$$
\begin{proposition}\label{prop:count}
For the collection $\{\mathcal{E}_l\}$ constructed above, there exists $\eta>0$ (depending only on $\delta,\tau, \bs{w}$ and the choice of $\varepsilon$ in Lemma \ref{lem:newlem}) such that for every $0<\gamma<1$ and $k\geq k(\gamma)$ it holds that
\begin{equation*}
\# \mc E_{n_k}\geq \left(\rho(1-\rho^{-\eta})\right)^{n_k(1-\gamma)}.
\end{equation*}
Moreover, for every $n_{k}<n\leq n_{k+1}$
\begin{equation}
\label{eq:card2}
\# \mc E_{n}\geq \#\mc {E}_{n_k}\cdot g_k(n),    
\end{equation}
where 
$$g_k(n):=\begin{cases}
1 & \mbox{if }n_k < n\le n_k^{(d)}\\[4pt]
\dfrac{1}{2}\prod_{i=1}^{d} \floor*{\rho_i^{(n_k^{(d)}+1-n_k^{(i)})}} \big( \rho(1-\rho^{-\eta})\big)^{n-(n_k^{(d)}+1)} & \mbox{if }n_k^{(d)}+1 \leq n\leq n_{k+1}
\end{cases}.$$
\end{proposition}

\begin{proof}
By choosing $\eta>0$ to be small enough such that $\rho^{-\varepsilon/2}+2^d \floor{\rho_1}^{-1}\leq \rho^{-\eta}$, it follows from Lemmas \ref{lem:case1}, \ref{lem:case2}, \ref{lem:case3} and \ref{lem:case4} that for every $E\in \mathcal{E}_{l-1}$ and every $k\in \N$, we have:\vspace{2mm}
\begin{enumerate}
    \item $\#\mathcal{E}_l(E)\geq \rho(1-\rho^{-\eta})$ if ${\xi} n_{k-1}^{(d)}<l\leq n_k$. \\
    \item $\#\mathcal{E}_l(E)=1$ if $n_k< l \leq n_k^{(d)}$. \\
    \item  $\#\mathcal{E}_l(E)\geq \dfrac{1}{2}\prod_{i=1}^{d} \floor*{\rho_i^{(n_k^{(d)}+1-n_k^{(i)})}}$ if $l=n_k^{(d)}+1$.\\
    \item $\#\mathcal{E}_l(E)\geq \rho(1-\rho^{-\eta})$ if $n_k^{(d)}+2\leq l\leq {\xi} n_k^{(d)}$.\vspace{2mm}
\end{enumerate}
This immediately implies \eqref{eq:card2}. Moreover, \eqref{eq:card2} implies that
\begin{multline*}
\# \mc E_{n_k}\geq \#\mc E_0\cdot \prod_{h<k}g_{h}(n_{h+1})\geq \frac{\# \mc E_0}{2^k}\prod_{h<k}\big( \rho(1-\rho^{-\eta})\big)^{n_{h+1}-(n_h^{(d)}+1)} \\
\geq \frac{1}{2^k}\big( \rho(1-\rho^{-\eta})\big)^{n_{k}-(n_{k-1}^{(d)}+1)}\geq  \big( \rho(1-\rho^{-\eta})\big)^{n_{k}-2n_{k-1}^{(d)}},
\end{multline*}
where we used the fact that
$$2^{k}\leq\left(\frac{\rho}{2}\right)^{n_{k-1}^{(d)}-1}.$$
The conclusion follows from the fact that $n_{k-1}/n_{k}\to 0$ as $k\to \infty$.
 \end{proof}

We also need an upper bound on the number of boxes in the collection $\mc E_n$ that intersect a given hypercube $B$ non-trivially.

\begin{proposition}
\label{prop:intersection}
Let $B$ be a box of side length $l$. For every $k\in \N$ and $n_k<n\leq  n_{k+1}$ it holds that
\begin{equation}
\label{eq:prodest}
\#\{E\in\mc E_n:E\cap B\neq \emptyset\} \\
\leq \prod_{i=1}^d \left(2\max\left\{\frac{\ell}{\rho_0^{(i)}\rho_i^{-n_k}},1\right\}\cdot {h^{(i)}_k}(n)\right),
\end{equation}
where
$${h^{(i)}_k}(n):=\begin{cases}
1 & \mbox{if }n_k < n\le n_k^{(d)}\\[4pt]
\min \left\{\dfrac{l}{\rho_0^{(i)}\rho_i^{-n_k^{(i)}}},1\right\} \cdot \max\left\{\dfrac{\rho_i^{-n_k^{(i)}}}{\rho_i^{-n}},\dfrac{\rho_0^{(i)}\rho_i^{-n_k^{(i)}}}{l}\right\} & \mbox{if }n_k^{(d)}+1 \leq n\leq n_{k+1}
\end{cases}.$$
\end{proposition}

\begin{proof}
It will be convenient to argue in each coordinate separately. Denoting by $\sigma_i$ the projection onto the $i$th coordinate, we have
\begin{equation*}
\#\{E\in\mc E_n:E\cap B\neq \emptyset\}    
\leq \prod_{i=1}^d \#\{\sigma_i(E):E\in\mc E_n\mbox{ and }\sigma_i(E)\cap\sigma_i(B)\neq\emptyset\}.
\end{equation*}
Fix $1\leq i\leq d$. We will show that
\begin{equation}
\label{eq:toshow}
\#\{\sigma_i(E):E\in\mc E_n\mbox{ and }\sigma_i(E)\cap\sigma_i(B)\neq\emptyset\} 
\leq 2\max\left\{\frac{\ell}{\rho_0^{(i)}\rho_i^{-n_k}},1\right\}\cdot {h^{(i)}_k}(n).
\end{equation}
We distinguish three cases:\\

\bigskip

\textbf{Case 1}: $\ell>\rho_0^{(i)} \rho_i^{-n_k}$.

\bigskip

In this case, 
$$\#\{\sigma_i(E):E\in\mc E_{n_k}\mbox{ and }\sigma_i(E)\cap\sigma_i(B)\neq\emptyset\} \leq \frac{l}{\rho_0^{(i)}\rho_i^{-n_k}}+2\leq \frac{3l}{\rho_0^{(i)}\rho_i^{-n_k}} $$
Since $n>n_k$, and due to the nestedness of the collections $\mathcal{E}_l$, we have
\begin{multline*}
\#\{\sigma_i(E):E\in\mc E_n\mbox{ and }\sigma_i(E)\cap\sigma_i(B)\neq\emptyset\}\\
\leq 3\frac{\ell}{\rho_0^{(i)}\rho_i^{-n_k}}\cdot \max_{E\in\mc E_{n_k}}\#\{\sigma_i(E')\subset\sigma_i(E):E'\in\mc E_n\}\leq 3\frac{\ell}{\rho_0^{(i)}\rho_i^{-n_k}}\cdot {f^{(i)}_k}(n),    
\end{multline*}
where 
$${f^{(i)}_k}(n):=\begin{cases}
1 & \mbox{if }n_k < n\le n_k^{(d)}\\[4pt]
\rho_i^{\,n-n_k^{(i)}} & \mbox{if }n_k^{(d)}+1 \leq n\leq n_{k+1}
\end{cases}.$$
Note that in this range of $\ell$ one has that ${f^{(i)}_k}(n)\leq {h^{(i)}_k}(n)$.

\bigskip 

\textbf{Case 2}: $\rho_0^{(i)}\rho_i^{-n_k^{(i)}}<\ell\leq \rho_0^{(i)}\rho_i^{-n_k}$.

\bigskip

Since $\ell\leq \rho_0^{(i)}\rho_i^{-n_k}$, $\sigma_i(B)$ intersects at most $2$ intervals in the level $n_k$ in the $i$-th coordinate. Hence,
\begin{multline*}
\#\{\sigma_i(E):E\in\mc E_n\mbox{ and }\sigma_i(E)\cap\sigma_i(B)\neq\emptyset\}\\
\leq 2 \cdot \max_{E\in\mc E_{n_k}}\#\{\sigma_i(E')\subset\sigma_i(E):E'\in\mc E_n\}=2\cdot {f^{(i)}_k}(n).    
\end{multline*}
Once again, in this range of $\ell$ one has that ${f^{(i)}_k}(n)\leq {h^{(i)}_k}(n)$.

\bigskip

\textbf{Case 3}: $ l \leq
\rho_0^{(i)}\rho_i^{-n_k^{(i)}}$. 

\bigskip

In this case, we have that
\begin{equation}
 \#\{\sigma_i(E):E\in\mc E_n\mbox{ and }\sigma_i(E)\cap\sigma_i(B)\neq\emptyset\} \leq 2\tilde f_k^{(i)}(n),     
\end{equation}
where 
$$\tilde f_k^{(i)}(n):=\begin{cases}
1, & n_k < n\le n_k^{(d)}\\[4pt]
\max\left\{\dfrac{l}{\rho_0^{(i)}\rho_i^{-n}},1\right\}, & n_k^{(d)}+1 \leq n\leq n_{k+1}
\end{cases}.$$
Let us explain why this holds. If $n_k < n\le n_k^{(d)}$, we have that $\sigma_i(B)\leq \sigma_i(E')$ for any $E'\in\mc E_n$, so that $\sigma_i(B)$ may intersect at most two projections $\sigma_i(E')$. In the case $n_k^{(d)}+1 \leq n\leq n_{k+1}$, we are simply applying the trivial upper bound
$$\max\left\{\frac{\ell}{\sigma_i(E')},1\right\}.$$
Note that in this range of $\ell$ one has that $\tilde f_k^{(i)}(n)\leq {h^{(i)}_k}(n)$.

\end{proof}
\subsection{Dimension estimate} Define $\mu$ recursively in the following way:\\
For $E_0\in \mathcal{E}_0$, set $\mu(E_0)=(\# \mathcal{E}_0)^{-1}$. For $n\geq 1$, and $E_n\in \mathcal{E}_n$, let
\begin{equation}\label{eq:meas}
  \mu(E_n)=\frac{1}{\#\mathcal{E}_n(E_{n-1})}\cdot \mu(E_{n-1}),  
\end{equation}
where $E_{n-1}$ is the unique box of level $n-1$ containing $E_n$. Then, by \cite[Proposition 1.7]{Fal14}, $\mu$ can be extended to a probability measure supported on $E_{\infty}$ such that for any Borel set $B$,
\begin{equation}\label{eq:muB}
    \mu(B)\leq \inf_{n\in \N}\sum_{\{E\in \mathcal{E}_n: E\cap B\neq \emptyset\}}\mu(E)\overset{\eqref{eq:meas}}{=} \inf_{n\in \N}\frac{\#\{E\in \mathcal{E}_n: E\cap B\neq \emptyset\}}{\# \mathcal{E}_n}. 
\end{equation}
Let $B$ be a box of side length $l$. Our goal is to provide a lower bound on the quantity
$$\log_{\ell}\mu(B)$$
for small enough $l$, this will give a lower bound on the Hausdorff dimension, according to Lemma \ref{lem:mdp}. Fix $\gamma>0$ and assume $l <\rho_1^{-n_{k(\gamma)}}$ (where $k(\gamma)$ is determined in proposition \ref{prop:count}). Let $n \geq n_{k(\gamma)}$ be such that 
\begin{equation}\label{eq:l}
    \rho_1^{-(n+1)}\leq \ell<\rho_1^{-n}
\end{equation}
and let $k$ be such that $$n_k<n\leq n_{k+1}.$$
Define $$n_B:=\max\{n,n_{k}^{(d)}+1\}$$
By \eqref{eq:muB}, we have
$$\mu(B)\leq \frac{\#\{E\in \mathcal{E}_{n_B}: E\cap B\neq \emptyset\}}{\# \mathcal{E}_{n_B}}$$
Hence
\begin{align}
    \log_{\ell}\mu(B)&\geq \log_{\ell}\left(\#\{E\in\mc E_{n_B}:E\cap B\neq \emptyset\}\right)+\log_{\ell}\left(\frac{1}{\#\mc E_{n_B}}\right) \notag\\
&=\frac{\log(\#\mc E_{n_B})}{|\log\ell|}-\frac{\log\left(\#\{E\in\mc E_{n_B}:E\cap B\neq \emptyset\}\right)}{|\log\ell|} \notag\\
&\overset{\eqref{eq:l}}{\geq} \frac{\log(\#\mc E_{n_B})}{(n+1)\log\rho_1}-\frac{\log\left(\#\{E\in\mc E_{n_B}:E\cap B\neq \emptyset\}\right)}{n\log\rho_1}. \label{eq:logmuB}
\end{align}
We will treat these two terms separately, using Propositions \ref{prop:count} and \ref{prop:intersection}. Let us start from the first one.
Since $n_B\geq n_k^{(d)}+1$, by Proposition \ref{prop:count}, we have that
$$\# \mc E_{n_B}\geq \dfrac{1}{2}\left(\rho(1-\rho^{-\eta})\right)^{n_k(1-\gamma)+n_B-(n_k^{(d)}+1)}\prod_{i=1}^{d} \floor*{\rho_i^{(n_k^{(d)}+1-n_k^{(i)})}} $$
Therefore,
\begin{multline}\label{eq:ing1}
    \frac{\log(\#\mc E_{n_B})}{(n+1)\log \rho_1}
\geq \sum_{i=1}^d \frac{n_k+n_B-n_k^{(i)}}{n}\frac{\log \rho_i}{\log \rho_1} \\
+O_d\left(\frac{1+|\log(1-\rho^{-\eta})|(n_k+n_B-(n_k^{(d)}+1))}{n\log \rho_1}+\frac{(\gamma\log \rho) n_k}{n\log \rho_1}+\frac{\log(\#\mc E_{n_B})}{n^2\log \rho_1}\right).
\end{multline}
Let us now move on to the second term in \eqref{eq:logmuB}. By Proposition \ref{prop:intersection} and the fact that $n_B\geq n_{k}^{(d)}+1$, we have that
\begin{multline*}
    \#\{E\in\mc E_{n_B}:E\cap B\neq \emptyset\}\leq \\
    \prod_{i=1}^d \left(2\max\left\{\frac{\ell}{\rho_0^{(i)}\rho_i^{-n_k}},1\right\}\cdot \min \left\{\dfrac{l}{\rho_0^{(i)}\rho_i^{-n_k^{(i)}}},1\right\} \cdot \max\left\{\dfrac{\rho_i^{-n_k^{(i)}}}{\rho_i^{-n_B}},\dfrac{\rho_0^{(i)}\rho_i^{-n_k^{(i)}}}{l}\right\}\right)\\
    \leq \prod_{i=1}^d \left(2\max\left\{\frac{\rho_1^{-n}}{\rho_0^{(i)}\rho_i^{-n_k}},1\right\}\cdot \min \left\{\dfrac{\rho_1^{-n}}{\rho_0^{(i)}\rho_i^{-n_k^{(i)}}},1\right\} \cdot \rho_i^{n_B-n_k^{(i)}}\right),
\end{multline*}
where the last inequality follows from the fact that $\rho_1^{-(n+1)}\leq l<\rho_1^{-n}$, $n_B\geq n$, and $\rho_0^{(i)}\cdot \rho_1\leq 1$.
Hence
\begin{multline}
\label{eq:ing2}
\frac{\log(\#\{E\in\mc E_{n_B}:E\cap B\neq \emptyset\})}{n\log\rho_1}\leq \sum_{i=1}^d\max\left\{\frac{n_k}{n}\frac{\log(\rho_i)}{\log(\rho_1)}-1,0\right\}\\
+\min\left\{\frac{n_k^{(i)}}{n}\frac{\log(\rho_i)}{\log(\rho_1)}-1,0\right\}+\frac{n_B-n_k^{(i)}}{n}\frac{\log \rho_i}{\log \rho_1}+O_{d}\left(\sum_{i=1}^d\frac{1+|\log \rho_0^{(i)}|}{n\log \rho_1 }\right).     
\end{multline}
Combining \eqref{eq:ing1} and \eqref{eq:ing2} we deduce that
\begin{multline*}
\log_{\ell}\mu(B)\geq \sum_{i=1}^d \left(\frac{n_k}{n}\frac{\log \rho_i}{\log \rho_1}-\max\left\{\frac{n_k}{n}\frac{\log(\rho_i)}{\log(\rho_1)}-1,0\right\}
-\min\left\{\frac{n_k^{(i)}}{n}\frac{\log(\rho_i)}{\log(\rho_1)}-1,0\right\}\right) \\
+O_{d,\tau,\bs w,\delta}\left(\frac{1}{\log R}+\gamma+\frac{1}{n}\right),
\end{multline*}
which can be re-written as
\begin{multline}
\label{eq:mu}
\log_{\ell}\mu(B)\geq\sum_{i=1}^d\min\left\{\frac{\log(\rho_i)}{\log(\rho_1)}\frac{n_k}{n},1\right\}+ \max\left\{1-\frac{\zeta_i\cdot \log(\rho_i)}{\log(\rho_1)}\frac{n_k}{n},0\right\} \\
+O_{d,\tau,\bs w,\delta}\left(\frac{1}{\log R}+\gamma+\frac{1}{n}\right).
\end{multline}
where $\zeta_i:=(1+\tau w_i)/(1+\tilde{w}_i)$.

\subsection{Finding the Minimum}
We will now study the function
$$f(x):=\sum_{i=1}^d\min\left\{\frac{\log(\rho_i)}{\log(\rho_1)}x,1\right\}+\max\left\{1-\frac{\zeta_i\cdot \log(\rho_i)}{\log(\rho_1)}x,0\right\}$$
for $x\in[0,1]$. For $h=1,\dotsc,d$ and $k=0,\dotsc,d$ we define
$$f_{h,k}(x):=k+d-h+\left(\log(\rho_1\dotsm\rho_h)-\sum_{i=1}^k\zeta_i\cdot\log(\rho_i)\right)\frac{x}{\log(\rho_1)},$$
where the sum over $i$ is null if $k=0$. Our goal will be to prove the following.
\begin{proposition}
\label{prop:min}
It holds that
\begin{equation*}
\min_{x\in[0,1]}f(x)=\min_{1\leq k\leq d}\min_{1\leq h\leq d}f_{h,k}\left(\frac{\log(\rho_1)}{\zeta_{k}\cdot\log(\rho_k)}\right).
\end{equation*}    
\end{proposition}

To prove Proposition \ref{prop:min}, it will be convenient for us to work with the intervals
\begin{equation*}
I_h:=\left(\frac{\log(\rho_1)}{\log(\rho_{h+1})},\frac{\log(\rho_{1})}{\log(\rho_h)}\right]\quad\mbox{and}\quad  J_k:=\left(\frac{\log(\rho_1)}{\zeta_{k+1}\cdot\log(\rho_{k+1})},\frac{\log(\rho_1)}{\zeta_{k}\cdot \log(\rho_{k})}\right] 
\end{equation*}
for $h=1,\dotsc,d$ and $k=0,\dotsc,d$. We adopt the convention that $\zeta_0:=1$, $\rho_0:=\rho_1$, and $\rho_{d+1}:=+\infty$. Note that both the families $I_h$ and $J_k$ form a partition of the interval $(0,1]$. 

\begin{lemma}
\label{lem:charoff}
If $x\in I_h\cap J_k$, then
\begin{equation}
\label{eq:hkequality}
f(x)=f_{h,k}(x)=k+d-h+\left(\log(\rho_1\dotsm\rho_h)-\sum_{i=1}^k\zeta_i\cdot\log(\rho_i)\right)\frac{x}{\log(\rho_1)},    
\end{equation}
where the sum over $i$ is null if $k=0$. Moreover, if $x\in J_k$ then
$$f(x)=\min_h f_{h,k}(x).$$   
\end{lemma}

\begin{proof}
For any $x\in I_h$ we have that
\begin{equation}
\label{eq:minh}
\sum_{i=1}^d\min\left\{\frac{\log(\rho_i)}{\log(\rho_1)}x,1\right\}=d-h+\sum_{i=1}^h\frac{\log(\rho_i)}{\log(\rho_1)}x,    
\end{equation}
while for any $x\in J_k$ we have that
\begin{equation}
\label{eq:maxk}
\sum_{i=1}^d\max\left\{1-\frac{\zeta_i\cdot \log(\rho_i)}{\log(\rho_1)}x,0\right\}=k-\sum_{i=1}^k\frac{\zeta_i\cdot \log(\rho_i)}{\log(\rho_1)}x.
\end{equation}
Hence, \eqref{eq:hkequality} follows. Further, it is clear that for any $h=1,\dotsc,d$ and any $x\in(0,1]$ it holds that 
$$d-h+\sum_{i=1}^{h}\frac{\log(\rho_i)}{\log(\rho_1)}x\geq \sum_{i=1}^d\min\left\{\frac{\log(\rho_i)}{\log(\rho_1)}x,1\right\},$$
so that for any $h,k$ and any $x\in(0,1]$
$$f_{h,k}(x)\geq \sum_{i=1}^d\min\left\{\frac{\log(\rho_i)}{\log(\rho_1)}x,1\right\}+k-\sum_{i=1}^k\frac{\zeta_i\cdot \log(\rho_i)}{\log(\rho_1)}x.$$
Taking the minimum in $h$, we conclude that for any $k$ and any $x\in(0,1]$
$$\min_{h}f_{h,k}(x)\geq \sum_{i=1}^d\min\left\{\frac{\log(\rho_i)}{\log(\rho_1)}x,1\right\}+k-\sum_{i=1}^k\frac{\zeta_i\cdot \log(\rho_i)}{\log(\rho_1)}x.$$
However, by \eqref{eq:minh}, if $x\in I_h$, the right-hand side is equal to $f_{h,k}$, so that the opposite inequality must also be true. Since both sides are independent of $h$ and the union of the intervals $I_h$ is the whole interval $(0,1]$, we conclude that for any $k$ and any $x\in (0,1]$ it holds that
\begin{equation}
\label{eq:interm}    
\min_{h}f_{h,k}(x)=\sum_{i=1}^d\min\left\{\frac{\log(\rho_i)}{\log(\rho_1)}x,1\right\}+k-\sum_{i=1}^k\frac{\zeta_i\cdot \log(\rho_i)}{\log(\rho_1)}x.
\end{equation}
Now, for any $x\in J_k$ we also have that
$$k-\sum_{i=1}^k\frac{\zeta_i\cdot \log(\rho_i)}{\log(\rho_1)}x=\sum_{i=1}^d\max\left\{1-\frac{\zeta_i\cdot \log(\rho_i)}{\log(\rho_1)}x,0\right\}.$$
Combining this with \eqref{eq:interm}, we deduce that for any $x\in J_k$ it must be
$$f(x)=\min_{h}f_{h,k}(x).$$
\end{proof}

\noindent
Let us now set
$$f_{k}(x):=\min_{h}f_{h,k}(x)$$
for $x\in(0,1]$ and $k=0,\dotsc,d$.

\begin{lemma}
\label{lem:slopes}
For any $k=0,\dotsc,d$ the function $f_k(x)$ is a piece-wise linear function whose slopes are in decreasing order. In particular, on any closed interval contained in $[0,1]$ the minimum of the function $f_k(x)$ is attained at the endpoints.
\end{lemma}

\begin{proof}
By \eqref{eq:interm}, on the interval $I_h$ we have that $f_{k}=f_{h,k}$. Now, $f_{h,k}$ is a line with slope
$$m_{h,k}:=\frac{1}{\log(\rho_1)}\left(\log(\rho_1\dotsm\rho_h)-\sum_{i=1}^k\zeta_i\cdot\log(\rho_i)\right).$$
Since $J_h$ is to the right of $J_{h+1}$ and $m_{h,k}\leq m_{h+1,k}$ the first statement is proved. The second statement is then a direct consequence of the first one. 
\end{proof}

We are now in a position to prove Proposition \ref{prop:min}. By Lemma \ref{lem:charoff}, we have that 
\begin{equation*}
\min_{x\in[0,1]}f(x)=\min_k \min_{x\in J_k}f(x)=\min_{k}\min_{x\in J_k}f_k(x).
\end{equation*}
Now, by Lemma \ref{lem:slopes}, we have that
$$\min_{x\in J_k}f_k(x)=\min\left\{f_{k}\left(\frac{\log(\rho_1)}{\zeta_{k}\cdot\log(\rho_k)}\right),f_{k}\left(\frac{\log(\rho_1)}{\zeta_{k+1}\cdot\log(\rho_{k+1})}\right)\right\}.$$
Since
$$f_{k}\left(\frac{\log(\rho_1)}{\zeta_{k+1}\cdot\log(\rho_{k+1})}\right)=f_{k+1}\left(\frac{\log(\rho_1)}{\zeta_{k+1}\cdot\log(\rho_{k+1})}\right)$$
and $f(0)=f(1)=d$, we have that
$$\min_{x\in[0,1]}f(x)=\min_{k=1,\dotsc,d} f_{k}\left(\frac{\log(\rho_1)}{\zeta_{k}\cdot\log(\rho_k)}\right).$$
In particular, we can always exclude the fact that the minimum is $d$ because the Hausdorff dimension of the set $W_{\bs w}(\tau)$ for $\tau>1$ is strictly less than $d$.
This proves the claim.

\subsection{Conclusion}

By Proposition \ref{prop:min} and \eqref{eq:mu}, we have that the Cantor set constructed in Subsection \ref{sec:3.2} has Hausdorff dimension bounded below by

$$\min_{h,k}\sum_{i=1}^h\frac{1+\tilde{w}_i}{1+\tau w_k}+d-h+k-\sum_{i=1}^{k}\frac{1+\tau w_i}{1+\tau w_k}+O_{d,\tau,\bs w,\delta}\left(\frac{1}{\log R}+\gamma\right),$$
where we replaced $n^{-1}$ by $n_0^{-1}\leq (\log R)^{-1}$. Since $R$ and $\gamma$ are arbitrary, the set $E_{\bs w}(\tau)$ has dimension at least
$$\min_{h,k}\sum_{i=1}^h\frac{1+\tilde{w}_i}{1+\tau w_k}+d-h+k-\sum_{i=1}^{k}\frac{1+\tau w_i}{1+\tau w_k}.$$

Now, there are two cases: either the minimum is realized for $k\leq K$ or for $k>K$, where $K$ is defined in Lemma \ref{lem:Rynne}. Let us assume first that the minimum is realized for $k\leq K$. Then, since $\tilde w_i\leq \tau w_{i}$ it must be $h\geq k$ and we have that
\begin{multline*}
\sum_{i=1}^h\frac{1+\tilde{w}_i}{1+\tau w_k}+d-h=\sum_{i=1}^k\frac{1+\tilde{w}_i}{1+\tau w_k}+\sum_{i=k+1}^h\frac{1+\tilde{w}_i}{1+\tau w_k}+d-h \\
\geq \sum_{i=1}^k\frac{1+\tau w_i-\delta(1+\tau w_i)}{1+\tau w_k}+\sum_{i=k+1}^h\frac{1+\tau w_{k}-\delta(1+\tau w_k)}{1+\tau w_k}+d-h \\
\geq \sum_{i=1}^k\frac{1+\tau w_i}{1+\tau w_k}+d-k-\delta(d+\tau).
\end{multline*}
Hence, the dimension is at least $d-\delta(d+\tau)$. If we assume that
$$\delta(\tau +d)<d-\dimH W_{\bs w}(\tau),$$
we find a contradiction. Hence, the minimum must be realized for $k>K$. Since $1+\tilde w_k\leq 1+\tau w_k$ and $\tilde w_i=\tilde w_j$ for any $i,j>K$, it must be $h=d$. The expression above is therefore equal to
$$\min_k\frac{d+1-\sum_{i=1}^k\tau w_k-\tau w_i}{1+\tau w_k},$$
which, by Theorem \ref{thm:Rynne}, yields Theorem \ref{thm:mainthm}.
 \bibliographystyle{alpha}
\bibliography{References}
\end{document}